\documentclass[12pt, reqno]{amsart}
\usepackage{epsfig}
\usepackage{amsmath}
\usepackage{amssymb}
\usepackage{amscd}
\usepackage{graphicx}

\theoremstyle{definition}
\newtheorem{thm}{Theorem}[section]
\newtheorem{lem}[thm]{Lemma}

\newtheorem{prop}[thm]{Proposition}
\newtheorem{ques}[thm]{Question}
\newtheorem{cor}[thm]{Corollary}
\newtheorem{de}[thm]{Definition}
\newtheorem{ex}[thm]{Example}
\newtheorem{conj}[thm]{Conjecture}
\newtheorem{rem}[thm]{Remark}

\def \ra {\rightarrow}

\def \N {\mathbb N}

\def \Z {\mathbb Z}

\def \B {\mathcal{B}}

\def \F {\mathcal{F}}

\def \J {\mathcal{J}}
\def \P {\mathcal{P}}

\def \U {\mathcal{U}}

\def \Ind {{\rm Ind}}
\def \pubd {{\rm pubd}}
\def \pud {{\rm pud}}
\def \pd {{\rm pd}}
\def \ip {{\rm ip}}
\def \ps {{\rm ps}}
\def \rs {{\rm rs}}
\def \rms {{\rm s}}
\def \cen {{\rm cen}}
\def \bfA {{\bf A}}
\def \bfx {{\bf x}}
\def \top {{\rm top}}
\def \IP {{\rm IP}}
\def \QR {{\rm QR}}
\def \Rec {{\rm Rec}}
\def \rM {{\rm M}}
\def \AP {{\rm AP}}
\def \supp {{\rm supp}}
\def \IE {{\rm IE}}
\def \diam {{\rm diam}}

\begin{document}
\title [family independence] {Family independence for
topological and measurable dynamics}
\author{Wen Huang, Hanfeng Li and Xiangdong Ye}

\address{\hskip-\parindent
Wen Huang, Department of Mathematics, University of Science
and Technology of China, Hefei, Anhui 230026, P.R. China}
\email{wenh@mail.ustc.edu.cn}

\address{\hskip-\parindent
Hanfeng Li, Department of Mathematics, SUNY at Buffalo,
Buffalo NY 14260-2900, U.S.A.}
\email{hfli@math.buffalo.edu}

\address{\hskip-\parindent
Xiangdong Ye, Department of Mathematics, University of Science
and Technology of China, Hefei, Anhui 230026, P.R. China}
\email{yexd@ustc.edu.cn}

\date{August 24, 2010}

\begin{abstract}
For a family $\F$ (a collection of subsets of $\Z_+$), the notion of $\F$-independence is defined both
for topological dynamics (t.d.s.) and measurable dynamics (m.d.s.).
It is shown that there is no non-trivial \{syndetic\}-independent m.d.s.;  a
m.d.s. is \{positive-density\}-independent if and only if it has completely
positive entropy; and a m.d.s. is weakly mixing if and only if it
is \{IP\}-independent.
For a t.d.s. it is proved that there is no non-trivial minimal
\{syndetic\}-independent system; a t.d.s. is weakly mixing if and only
if it is \{IP\}-independent.

Moreover, a non-trivial proximal topological K system is constructed, and a
topological proof of the fact that minimal topological K implies
strong mixing is presented.
\end{abstract}


\maketitle


\section{Introduction}
By a {\it topological dynamical system} (t.d.s.) $(X,T)$ we mean a
compact metrizable space $X$ together with a surjective continuous map
$T$ from $X$ to itself. For a t.d.s. $(X,T)$ and nonempty open
subsets $U$ and $ V$ of $X$ let $N(U,V)=\{n\in\Z_+: U\cap
T^{-n}V\not=\emptyset\}$, where $\Z_+$ denotes the set of non-negative
integers.
It turns out that many recurrence properties of t.d.s. can be
described using the return times sets $N(U,V)$, see \cite{A, F, G1,
HSY, HY1}. For example, for a t.d.s. $(X,T)$ it is known that $T$ is
(topologically) {\it strongly mixing} iff $N(U,V)$ is cofinite, $T$
is (topologically) {\it weakly mixing} iff $N(U,V)$ is thick
\cite{F} and $T$ is (topologically) {\it mildly mixing} iff $N(U,V)$
is an $(\IP-\IP)^*$ set \cite{HY1, GW}, for each pair of nonempty
open subsets $U$ and $V$. Huang and Ye \cite{HY1} showed that a
minimal system $(X,T)$ is weakly mixing iff the lower Banach density
of $N(U,V)$ is $1$, and $(X,T)$ is mildly mixing iff $N(U,V)$ is an
$\IP^*$ set, for each pair of nonempty open sets $U$ and $V$.

By a {\it measurable dynamical system} (m.d.s.) we mean a quadruple
$(X,{\mathcal{B}},\mu,T)$, where $(X, \B, \mu)$ is a Lebesgue space
(i.e., $X$ is a set, $\mathcal{B}$ is the
$\sigma$-algebra of Borel subsets on $X$ for some Polish topology on $X$,
and $\mu$ is a probability measure on $\B$) and
$T:X\rightarrow X$ is measurable and measure-preserving, that
is: $\mu(B)=\mu (T^{-1}B)$ for each $B\in \mathcal{B}$. For a
t.d.s $(X,T)$, there are always invariant Borel probability measures on
$X$ and thus for each such measure $\mu$, $(X,\mathcal{B}_X,\mu,T)$,
with $\mathcal{B}_X$ the Borel $\sigma$-algebra on $X$, is a m.d.s..
For a m.d.s. $(X, \B, \mu, T)$, let
${\mathcal B}^+=\{B\in {\mathcal B}:\mu(B)>0\}$ and
$N(A,B)=\{n\in{\Bbb Z}_+: \mu(A\cap T^{-n}B)>0\}$ for $A, B\in
\mathcal{B}^+$.  It is known that $T$ is ergodic iff
$N(A,B)\not=\emptyset$ iff $N(A,B)$ is syndetic; $T$ is weakly
mixing iff the lower Banach density of $N(A,B)$ is $1$ iff $N(A,B)$
is thick; and $T$ is mildly mixing iff $N(A,B)$ is an $\IP^*$ set iff
$N(A,B)$ is an $(\IP-\IP)^*$  set for all $A,B\in \mathcal{B}^+$
iff for each $\IP$ set $F$ and $A\in{\mathcal B}^+$,
$\mu(\bigcup_{n\in F}T^{-n}A)=1$. Finally, it is known that $T$ is
intermixing iff $N(A,B)$ is cofinite for all $A,B\in \mathcal{B}^+$,
see \cite{KY, KY1} and references therein.

In ergodic theory there exists a rich and powerful entropy theory.
The analogous notion of topological entropy
was introduced soon after the measure theoretical one, and was
widely studied and applied. Notwithstanding, the level of
development of topological entropy theory lagged behind. In recent
years however this situation is rapidly changing. A turning point
occurred with F. Blanchard's pioneering papers \cite{B1, B2} in
the 1990's.

In recent years a local entropy theory has been developed, see
\cite{GY} for a survey. More precisely, in \cite{B1} Blanchard
introduced the notions of completely positive entropy (c.p.e.) and
uniformly positive entropy (u.p.e.) as topological analogues of the
$K$-property in ergodic theory. In \cite{B2} he defined the notion
of entropy pairs and used it to show that a u.p.e. system is
disjoint from all minimal zero entropy systems. The notion of
entropy pairs can also be used to show the existence of the maximal
zero entropy factor for any t.d.s., namely the topological Pinsker
factor \cite{BL}. Blanchard et al. \cite{BHMMR} also introduced the
notion of entropy pairs for an invariant Borel probability measure.
Glasner and Weiss \cite{GW2} introduced the notion of entropy
tuples. In order to gain a better understanding of the topological
version of a K-system, Huang and Ye \cite{HY2} introduced the notion
of entropy tuples for an invariant Borel probability measure.
They showed that if
$(X,T)$ is a t.d.s. and $k\ge 2$,
then a non-diagonal tuple $(x_1, \dots, x_k)$ in $X^k$ is an
entropy tuple iff for every choice of neighborhoods $U_i$ of $x_i$
there is a subset $F$ of $\Z_+$ with positive
density such that $\bigcap_{i\in F}
T^{-i}U_{s(i)}\not=\emptyset$ for each $s\in
\{1,\ldots,k\}^F$. We mention that at the same time a theory of
sequence entropy tuples and tame systems were developed \cite{G06, Huang06, G07}. It is Kerr
and Li who captured the idea behind the results on entropy tuples,
sequence entropy tuples and tame systems and treated them systematically
using a notion called independence in \cite{KL2, KL3}, which first appeared
in Rosenthal's proof of his groundbreaking $\ell_1$ theorem
\cite{Ros74, Ros78}.

Let $(X,T)$ be a t.d.s.. For a tuple $\bfA=(A_1,\dots,A_k)$ of
subsets of $X$, we say a subset $F\subseteq \Z_+$ is an {\it
independence set} for $\bfA$ if for any nonempty finite subset
$J\subseteq F$, we have $$\bigcap_{j\in
J}T^{-j}A_{s(j)}\not=\emptyset$$ for any $s\in \{1,\dots,k\}^J$. We
call a tuple $\bfx=(x_1,\dots,x_k)\in X^k$ (1) an {\it
IE-tuple} if for every product neighborhood $U_1\times \dots\times
U_k$ of $\bfx$ the tuple $(U_1,\dots,U_k)$ has an independence
set of positive density; (2) an {\it IT-tuple} if for every product
neighborhood $U_1\times \dots\times U_k$ of $\bfx$ the tuple
$(U_1,\dots,U_k)$ has an infinite independence set; (3) an {\it
IN-tuple} if for every product neighborhood $U_1\times \dots\times
U_k$ of $\bfx$ the tuple $(U_1,\ldots,U_k)$ has arbitrarily
long finite independence sets. Kerr and Li \cite{KL2} showed that (1)
entropy tuples are exactly non-diagonal IE-tuples; (2)
sequence entropy tuples are exactly  non-diagonal IN-tuples, and  in
particular a t.d.s. $(X,T)$ is null iff it has no non-diagonal IN-pairs;
(3) a t.d.s. $(X,T)$ is tame iff it has no non-diagonal IT-pairs.
For similar results concerning m.d.s. see \cite{KL3}.

Thus the notion of independence is very useful to describe dynamical
properties. A family is a collection of subsets of the set of nonnegative integers satisfying suitable hereditary property.
For a family $\F$, the notion of $\F$-independence can
be defined both for topological dynamics (t.d.s.) and measurable
dynamics (m.d.s.). For instance, in the topological case, we say that $(X, T)$ is $\F$-independent if every finite tuple of subsets of $X$ with nonempty interiors
has an independence set in $\F$.
So a natural question is: for a given family $\F$
which dynamical property is equivalent to $\F$-independence? In this
paper we try to answer this question.

It is shown that there is no non-trivial \{syndetic\}-independent m.d.s.;  a
m.d.s. is \{positive-density\}-independent iff it has completely positive entropy;
and a m.d.s. is weakly mixing iff it is
\{infinite\}-independent iff  it is \{IP\}-independent.
For a t.d.s. it is proved that there is no non-trivial minimal
\{syndetic\}-independent system; a t.d.s. is weakly mixing iff it is
\{infinite\}-independent iff it is \{IP\}-independent.

Moreover, a non-trivial proximal topological K system (see Definition~\ref{topological K:def} below) is constructed, and a
topological proof (using independence) of the fact that
minimal topological K implies strong mixing is presented. In a
forthcoming paper \cite{HLiY} we will deal with the problem of how to
localize the notion of $\F$-independence.

In \cite{B2} Blanchard raised the question whether there exists
any non-trivial minimal uniformly positive entropy (equivalently,
\{positive-density\}-independent of order $2$ in our terminology)
t.d.s.. This was answered affirmatively by Glasner and Weiss in
\cite{GW1}. Later Huang and Ye showed there are non-trivial
minimal \{positive-density\}-independent t.d.s. \cite{HY2}.
However, the constructions in \cite{GW1} and \cite{HY2} are based
on showing that any minimal topological model of a K-system is
such an example and then using the Jewett-Krieger  theorem to
obtain such a topological model. So far there is no explicit
topological construction of such examples. Since the family of
syndetic sets is  just slightly smaller than the family of
positive upper Banach density
 sets, our result of the non-existence of  non-trivial minimal \{syndetic\}-independent
 t.d.s. explains why it is so difficult to construct examples for Blanchard's question.

The paper is organized as follows. In Section 2 we investigate the
relationship between a given family $\F$ and the associated block
family $b\F$. In Section 3, the basic properties of
$\F$-independence for a t.d.s. are discussed. Particularly we show
that $\F$ and $b\F$ define the same notion of independence. In
Section 4, the basic properties of $\F$-independence for a m.d.s.
are discussed. In Section 5, we investigate classes of
$\F$-independent systems for t.d.s. and show that there is no
non-trivial minimal \{syndetic\}-independent t.d.s.. Moreover, a
non-trivial proximal topological K system is constructed. In
Section 6, we investigate classes of $\F$-independent systems for
m.d.s. and show that a m.d.s. is \{positive-density\}-independent
iff it has completely positive entropy. We also show that there is
no non-trivial \{syndetic\}-independent m.d.s.. In Section 7, we
give a topological proof of the fact that minimal topological K
implies strong mixing. An interesting combinatorial result, which
is needed for the proof of non-existence of no-trivial minimal
\{syndetic\}-independent t.d.s., is established in the Appendix.

Throughout this paper, we use $\Z_+$ and $\N$ to denote the sets of
nonnegative integers and positive integers respectively. For a
subset $F$ of $\Z$ and $m\in \Z$ we denote $\{j+m: j\in F\}$ by
$F+m$. For a subshift $X$ of $\{0, 1, \dots, k\}^{\Z_+}$ or $\{0, 1,
\dots, k\}^{\Z}$ and $a\in \{0, 1, \dots, k\}^{\{1, \dots, m\}}$ for
some $m\in \N$, we denote $\{x\in X: (x(0), x(1), \dots,
x(m-1))=a\}$ by $[a]_X$. For a t.d.s. $(X, T)$ and subsets $U,
V\subseteq X$, we denote by $N(U, V)$ the set $\{n\in \Z_+: U\bigcap
T^{-n}V\neq \emptyset\}$; for $x\in X$ we shall write $N(x, U)$ for
$N(\{x\}, U)$. For a m.d.s. $(X, \B, \mu, T)$ and $A, B\in \B$, we
denote by $N(A, B)$ the set $\{n\in \Z_+: \mu(A\bigcap
T^{-n}B)>0\}$.

\noindent{\it Acknowledgements.} W.H. and X.Y. are partially
supported by grants from NNSF of China (10531010, 11071231) and  973
Project (2006CB805903). W.H  is partially supported by NNSF of China
(10911120388), Fok Ying Tung Education Foundation, FANEDD (Grant
200520) and the Fundamental Research Funds for the Central
Universities. H.L. is partially supported by NSF grant DMS-0701414.
Part of this work was carried out during visits of H.L. to W.H. and
X.Y. in the summers of 2008 and 2009. H.L. is grateful to them for
their warm hospitality. We thank the referees for very helpful
comments.

\section{Preliminary}

The idea of using families to describe dynamical properties goes
back at least to Gottschalk and Hedlund \cite{GH}. It was developed
further by Furstenberg
\cite{F, HF}. For a systematic study and recent
results, see \cite{A, G1, HSY, HY1}.

Let us recall some notations related to a family (for details see
\cite{A}). Let $\P=\P({\Z}_{+})$ be the collection of all subsets of
$\Z_+$. A subset $\F$ of $\P$ is a {\it family}, if it is hereditary
upward. That is, $F_1 \subseteq F_2$ and $F_1 \in \F$ imply $F_2 \in
\F$. A family $\F$ is {\it proper} if it is a proper subset of $\P$,
i.e. neither empty nor all of $\P$. It is easy to see that $\F$ is
proper if and only if ${\Z}_{+} \in \F$ and $\emptyset \notin \F$.
Any subset $\mathcal{A}$ of $\P$ generates a family
$[\mathcal{A}]=\{F \in \P:F \supseteq A$ for some $A \in
\mathcal{A}\}$. If a proper family $\F$ is closed under taking finite
intersection, then $\F$ is called a {\it filter}. For a family $\F$,
the {\it dual family} is
$$\F^*=\{F\in\P: {\Z}_{+} \setminus F\notin\F\}=\{F\in \P:F \cap F' \neq
\emptyset \text{ for  all } F' \in \F \}.$$ $\F^*$ is a family, proper
if $\F$ is. Clearly,
$$(\F^*)^*=\F\ \text{and}\ {\F}_1\subseteq {\F}_2 \Longrightarrow
{\F}_2^* \subseteq {\F}_1^*.$$

There is an important property being well studied: the Ramsey
property. We say that a family $\F$ has the {\it Ramsey property} if
whenever $F_1\cup F_2\in\F$, one has either $F_1\in\F$ or $F_2\in\F$.
One can show that a proper family $\F$ has the Ramsey property if and only if $\F^*$ is a
filter \cite[page 26]{A}.

Denote by $\F_{\inf}$ the family of all infinite subsets of ${\Z}_{+}$
and by $\F_{\rm c}$ the dual family $\F^*_{\inf}$.
 Note that $\F_{\rm c}$ is the collection
of all cofinite subsets of $\Z_+$.
{\bf All the families considered in this paper are assumed to be proper and contained
in $\F_{\inf}$}.

Let $F$ be a subset of $\Z_+$. The {\it lower density} and {\it
upper density} of $F$ are defined by
$$\underline{d}(F)=\liminf_{n\to +\infty} \frac{1}{n}|F\cap \{0,1,\ldots,n-1\}|\
\text{and}\ \overline{d}(F)=\limsup_{n\to +\infty} \frac{1}{n}|F\cap
\{0,1,\ldots,n-1\}|.$$ If $\underline{d}(F)=\overline{d}(F)=d(F)$,
we then say that the {\it density} of $F$ is $d(F)$. The {\it upper
Banach density}
of $F$ is defined by
$$BD^*(F)=\limsup_{|I|\to +\infty}\frac{|S\cap
I|}{|I|},$$
where $I$ is taken over all nonempty finite intervals of
$\Z_+$.

We denote by $\F_{\pd}$ the family generated by
sets with positive density,
by $\F_{\pud}$ the family of sets
with positive upper density,
and by $\F_{\pubd}$ the family of sets
with positive upper Banach density.


Note that a subset $F$ of $\Z_+$ is said to be {\it thick} if for any $n\in
\N$ there exists some $m\in \Z_+$ such that
$\{m,m+1,\dots,m+n\}\subseteq F$. An infinite subset $F=\{s_1<s_2<\cdots\}$ of $\Z_+$ is said to
be {\it syndetic} if $\{s_{n+1}-s_n:n\in \N$\} is bounded.
A subset of $\Z_+$ is called {\it piecewise syndetic} if it is the intersection
of a thick set and a syndetic set.
We denote by
$\F_{\rm t}$,
$\F_{\rms}$ and $\F_{\ps}$  the families of
thick sets,
syndetic sets and piecewise syndetic sets respectively.

A subset $F$ of $\Z_+$ is called a {\it central set} if
there exists a t.d.s. $(X, T)$, a point $x\in X$, a minimal point
$y\in X$ which is proximal to $x$ and a neighborhood $U_y$ of $y$
such that $F\supseteq N(x, U_y)$ \cite[Section 8.3]{HF}. Here $y$ is proximal to $x$ means that
for a compatible metric $d$ of $X$, one has $\inf_{n\in \Z_+}d(T^nx, T^ny)=0$.
We denote by $\F_{\cen}$ the family of all central
sets.

A subset $F$ of $\Z_+$ is called an $\IP$-set if there exists a
sequence $\{a_n\}_{n\in \N}$ in $\N$ such that $F$ consists of
$a_{n_1}+a_{n_2}+\dots+ a_{n_k}$ for all $k\in \N$ and $n_1<n_2<\dots
<n_k$. We denote by $\F_{\ip}$ the family generated by all
$\IP$-sets.

\begin{de} \label{block:def}
Let $\F$ be a family. The {\it block family} of $\F$, denoted by
$b\F$, is the family consisting of sets $S\subseteq \Z_+$ for which
there exists some $F\in \F$ such that for every finite subset $W$ of
$F$ one has $m+W\subseteq S$ for some $m\in \Z$.
\end{de}

Clearly $\F\subseteq b\F$ and
$b(b\F)=b\F$. It is also clear that $b\F_{\inf}=\F_{\inf}$ and $b\F_{\rm c}=\F_{\rm t}$.


\begin{ex} \label{bF:ex}
It is clear that $b\F_{\pd}\subseteq b\F_{\pud}\subseteq \F_{\pubd}$. It is a result of
Ellis that $\F_{\pubd}\subseteq b\F_{\pd}$ \cite[Theorem 3.20]{HF} (one can also give
a topological proof for this, using an argument similar to that in the proof of
Lemma~\ref{independence density:lemma}). Thus one has $b\F_{\pd}=b\F_{\pud}=\F_{\pubd}$.
\end{ex}

\begin{ex} \label{s vs ps:ex}
It is clear that $b\F_{\rms}\subseteq \F_{\ps}$. Let $S_1\in \F_{\rm
t}$ and $S_2\in \F_{\rms}$. Then for each $n\in \N$ we can find some
$a_n\in \Z_+$ with $[a_n, a_n+n]\subseteq S_1$. Some subsequence of
the sequence $\{1_{([a_n, a_n+n]\cap S_2)-a_n}\}_{n\in \N}$
converges in $\{0, 1\}^{\Z_+}$ to $1_F$ for some subset $F$ of
$\Z_+$. It is easy to see that $F$ is syndetic and that for every
finite subset $W$ of $F$ one has $m+W\subseteq S_1\cap S_2$ for some
$m\in \Z_+$. Therefore $b\F_{\rms}\supseteq \F_{\ps}$, and hence
$b\F_{\rms}=\F_{\ps}$.
\end{ex}

\begin{ex} \label{central vs ps:ex}
It is clear that $\F_{\cen}\subseteq \F_{\ps}$ and hence
$b\F_{\cen}\subseteq b\F_{\ps}=\F_{\ps}$. Let $S\in \F_{\ps}$.
Denote by $X$ the smallest closed shift-invariant subset of $\{0, 1\}^{\Z}$
containing $1_S$.
Note that $S=N(1_S, [1]_X)$.
By \cite[Theorem 6]{BF} there is a minimal point $x$ of $X$ contained in $[1]_X$.
Say, $x=1_{S'}$. Set $F=S'\cap \Z_+$. Then $F=N(x, [1]_X)$ is central.
Since $x$ is in $X$, it is easy to see that for every
finite subset $W$ of $F$ one has $m+W\subseteq S$ for some
$m\in \Z$. This means that $S\in b\F_{\cen}$. Therefore $b\F_{\cen}\supseteq \F_{\ps}$, and hence
$b\F_{\cen}=\F_{\ps}$.
\end{ex}

The following result shows the relation between the block family and
the broken family introduced in \cite[Defintion 2]{BF}.

\begin{prop} \label{block family:prop}
Let $\F$ be a family. Let $S\subseteq \Z_+$. Then $S\in b\F$ if and
only if there exist an $F=\{p_1<p_2< \dots\}\in \F$
and a (not necessarily strictly) increasing sequence
$\{b_j\}_{j=1}^{\infty}$ of integers such that $S\supseteq
\bigcup_{j=1}^\infty\{b_j+\{p_1,p_2,\dots, p_j\}\}$.
\end{prop}
\begin{proof} The ``if" part is trivial.

Suppose that $S\in b\F$.
Let $F=\{p_1<p_2< \dots\}\in \F$ witnessing this. Then for each $j\in \N$
we find some $b_j\in \Z$ with $b_j+\{p_1, \dots, p_j\}\subseteq S$.
Note that $b_j+p_1\ge 0$ for every $j\in \N$. Thus we can find an
increasing subsequence $\{b_{j_k}\}^{\infty}_{k=1}$ of
$\{b_j\}^{\infty}_{j=1}$. Then for each $k\in \N$ we have $b_{j_k}+\{p_1,
\dots, p_k\}\subseteq b_{j_k}+\{p_1, \dots, p_{j_k}\}\subseteq S$. Thus
$S\supseteq \bigcup_{k=1}^\infty\{b_{j_k}+\{p_1,p_2,\cdots, p_k\}\}$.
This proves the ``only if" part.
\end{proof}

The next result follows from Proposition~\ref{same indep:prop}
and Lemma~\ref{dynamical Ramsey to Ramsey:lemma}, which we shall
prove in the next section.

\begin{prop}\label{block0}
If $\F$ has the Ramsey property, then so does $b\F$.
\end{prop}

We remark that if $b\F$ has the Ramsey property, it is not
necessarily true that $\F$ has the Ramsey property. For example,
$\F_{\pud}$ and $\F_{\pubd}$ have the Ramsey property, while $\F_{\pd}$ does not.

\medskip
For the readers' convenience we make the following table. All the
definitions of the families can be found in this section except
$\F_{\rm ss}$ and $\F_{\rm rs}$ which can be found in Section
\ref{c7} and Section \ref{c3} respectively.

\begin{table}[!h]
\caption{Notions for families} \vspace*{1.5pt}
\begin{center}
\begin{tabular}{| p{2.5cm} | p{1.8cm} | p{2.6cm} |p{2.0cm}|p{1.8cm}|p{2.4cm}| }\hline
\multicolumn{1}{|c|} {$\F_{\rm ss}$}  & \multicolumn{1}{|c|}
{$\F_{\inf}$} & \multicolumn{1}{|c|} {$\F_{\pubd}$} &
\multicolumn{1}{|c|} {$\F_{\ps}$} & \multicolumn{1}{|c|}
{$\F_{\rms}$} & \multicolumn{1}{|c|}
{$\F_{\pd}$} \\
\hline
   Section \ref{c7} & all infinite sets & all positive upper Banach
  density sets
& all piecewise syndetic sets & all syndetic sets & generated by all positive density sets\\
\hline \ \ \ \ \ \ $\F_{\rs}$   & \ \ \ \ \ \ $\F_{\rm c}$ & \ \ \ \
\ \ $\F_{\ip}$
& \ \ \ \ \ \ $\F_{\cen}$ & \ \ \ \ \ \ $\F_{\rm t}$ & \ \ \ \ \ \ $\F_{\pud}$\\
\hline  generated by $\{n\Z_+\!:n\in\!\N\}$ &all cofinite sets &
generated by all IP-sets
& all central sets & all thick sets & all positive upper density sets\\
\hline
 \end{tabular}
\end{center}
\end{table}

\section{Independence: topological case}\label{c3}

In this section, for a given family $\F$, we define
$\F$-independence for t.d.s., and discuss $1$-independence
for various families. Recall first the notion of
independence set introduced in \cite[Definition 2.1]{KL2}.

\begin{de}\label{indede}
Lt $(X,T)$ be a t.d.s..
For a tuple
$\bfA=(A_1,\ldots,A_k)$ of subsets of $X$, we say that a subset
$F\subseteq \Z_+$ is an {\it independence set} for $\bfA$ if for
any nonempty finite subset $J\subseteq F$, we have $$\bigcap_{j\in
J}T^{-j}A_{s(j)}\not=\emptyset$$ for any $s\in \{1,\dots,k\}^J$.
\end{de}

We shall denote the collection of all independence sets for $\bfA$ by
$\Ind(A_1,\ldots, A_k)$ or $\Ind \bfA$. The basic properties of independence sets are listed below.

\begin{lem}\label{ye1}
The following hold:
\begin{enumerate}
\item If $F\in \Ind(A_1, \dots, A_k)$ and $F_1\subseteq F$, then $F_1\in \Ind(A_1, \dots, A_k)$.

\item $F=\{a_1,a_2,\ldots\}$ is in $\Ind(A_1, \dots, A_k)$
if and only if $\{a_1,\ldots,a_n\}$ is in $\Ind(A_1, \dots, A_k)$ for
each $n\in \N$.

\item If $m\in \Z$ and $F, m+F\subseteq \Z_+$, then $F$ is in $\Ind(A_1, \dots, A_k)$
if and only if $m+F$ is so.

\item Let $F\subseteq \Z_+$ and $X$ be the subshift of
$\{0, 1\}^{\Z}$ generated by $\{1_E:E\subseteq F\}$. Then $F\in \Ind([0]_X, [1]_X)$.
\end{enumerate}
\end{lem}

\begin{de} \label{dynamical Ramsey property:def}
Let $\F$ be a family. We say that $\F$ has the {\it dynamical
Ramsey property}, if for any t.d.s. $(X,T)$, any $k\in \N$ and
closed subsets $A_1,A_2, \ldots, A_k, A_{1,1}, A_{1,2}$ of $X$ with
$A_1=A_{1,1}\cup A_{1,2}$, whenever $\Ind(A_1, A_2, \ldots, A_k)\cap
\F\not=\emptyset$, one has either $\Ind(A_{1,1}, A_2, \ldots, A_k)\cap
\F\not=\emptyset$ or $\Ind(A_{1,2}, A_2, \ldots, A_k)  \cap
\F\not=\emptyset$.
\end{de}

It was shown in \cite[Lemmas 3.8 and 6.3]{KL2} that the families $\F_{\pd}$ and $\F_{\inf}$ have the
dynamical Ramsey property.

Similar to the definition of u.p.e. of order $n$
(see \cite{HY2}),  we have

\begin{de} \label{tuple:def}
Let $\F$ be a family, $k\in\N$ and $(X,T)$ be a t.d.s..
A tuple $(x_1,\ldots,x_k)\in X^k$ is called an {\it $\F$-independent tuple}
if for any neighborhoods $U_1,\ldots,U_k$ of $x_1,\ldots,x_k$ respectively,
one has $\Ind(U_1,\ldots,U_k)\cap\F\not=\emptyset$. A t.d.s. is said to be
{\it $\F$-independent of order $k$}, if for each tuple of nonempty
open subsets $U_1,\ldots,U_k$,
$\Ind(U_1,\ldots,U_k)\cap\F\not=\emptyset$, and a t.d.s. is said to be
{\it $\F$-independent}, if it is {\it $\F$-independent} of order $k$
for each $k\in\N$.
\end{de}

Standard arguments as in \cite{B2} show the following:

\begin{prop} \label{dynamical Ramsey to local:prop}
Let $\F$ be a family with the dynamical Ramsey property, and let $(X, T)$ be a t.d.s..
The following are true:
\begin{enumerate}
\item If $\bfA=(A_1, \dots, A_k)$ is a tuple of closed subsets of $X$ with
$\Ind \bfA\cap \F\neq \emptyset$, then there exists $x_j\in A_j$ for
each $1\le j\le k$ such that $(x_1, \dots, x_k)$ is an
$\F$-independent tuple.

\item Let $k\in \N$. Then the set of  $\F$-independent $k$-tuples of $X$
is a closed $T\times \dots\times T$-invariant subset of $X^k$.

\item Let $(Y, S)$ be a t.d.s. and $\pi: X\rightarrow Y$ be a factor map,
i.e., $\pi$ is continuous surjective and equivariant. Let $k\in \N$.
Then $\pi\times \dots \times \pi$ maps the set of  $\F$-independent
$k$-tuples of $X$ onto the set of $\F$-independent $k$-tuples of
$Y$.
\end{enumerate}
\end{prop}

Recall that two t.d.s. $(X, T)$ and $(Y, S)$ are said to be {\it
disjoint} \cite{F} if $X\times Y$ is the only nonempty closed subset
$Z$ of $X\times Y$ satisfying $(T\times S)(Z)=Z$ and projecting
surjectively to $X$ and $Y$ under the natural projections $X\times Y
\rightarrow X$ and $X\times Y \rightarrow Y$ respectively. Following
the arguments in the proofs of \cite[Proposition 6]{B2} and
\cite[Theorem 2.1]{BL}  we have

\begin{thm} \label{disjoint:thm}
Let $\F$ be a family with the dynamical Ramsey property. The following are true:
\begin{enumerate}
\item Each t.d.s. which is $\F$-independent of order $2$ is disjoint from every minimal
system without non-diagonal $\F$-independent pairs.

\item Each t.d.s. admits a maximal factor with no non-diagonal $\F$-independent
pairs.
\end{enumerate}
\end{thm}

Different families might lead to the same notion of independence. In fact,
it follows from
Lemma~\ref{ye1}(2)(3) that $\Ind(A_1, \dots, A_k)\cap \F\neq
\emptyset$ if and only if $\Ind(A_1, \dots, A_k)\cap b\F\neq
\emptyset$. Thus we have:

\begin{prop} \label{same indep:prop}
Let $\F$ be a family. Then:
\begin{enumerate}
\item The families $\F$ and $b\F$ define the same
notion of independence.

\item
$\F$ has the dynamical Ramsey property if and only if $b\F$ does.
\end{enumerate}
\end{prop}

\begin{thm}\label{ye3} Let $\F_1, \F_2$ be two families having
the dynamical Ramsey property. Then each $\F_1$-independent pair is
an $\F_2$-independent pair and viceversa if and only if
$b\F_1=b\F_2$.
\end{thm}
\begin{proof} The ``if" part follows from Proposition~\ref{same indep:prop}.

Now assume that each $\F_1$-independent pair is an
$\F_2$-independent pair.
We are going to show that $b\F_1\subseteq b\F_2$.

Let $F\in \F_1$. Denote by $X$ the smallest closed shift-invariant
subset of $\{0, 1\}^{\Z}$ containing $\{1_E:E\subseteq F\}$. Then
$F\in \Ind([0]_X,[1]_X)$ and $$X=\overline {\{T^i1_E: i\in\Z,
E\subseteq F\}},$$ where $T$ denotes the shift. Since $\F_1$ has the
dynamical Ramsey property, there exists $(x,y)\in [0]_X\times [1]_X$
which is $\F_1$-independent. As each $\F_1$-independent pair is an
$\F_2$-independent pair, we get that $\Ind([0]_X, [1]_X)\cap
\F_2\not=\emptyset$. Let $F'\in \Ind([0]_X, [1]_X)\cap \F_2$. For any
finite subset $W$ of $F'$, there exists $x_W\in \bigcap_{k\in
W}T^{-k}([1]_X)$. Then $x_W(k)=1$ for every $k\in W$. Since $x_W\in
X$, it follows that there exists some $m\in \Z$ with $m+W\subseteq F$.
Therefore $F\in
b\F_2$. Thus $\F_1\subseteq b\F_2$, and hence $b\F_1\subseteq
b(b\F_2)=b\F_2$. This proves the ``only if" part.
\end{proof}

From Theorem~\ref{ye3} one sees that if a family $b\F$ has the
dynamical Ramsey property, then among the families which has the
dynamical Ramsey property and defines the same independence as $\F$
does, $b\F$ is the largest one.

\begin{lem} \label{dynamical Ramsey to Ramsey:lemma}
Let $\F$ be a family. If $\F$ has the Ramsey property, then for any
t.d.s. $(X, T)$ and closed subsets $Y, Y_1, Y_2$ of $X$ with
$Y=Y_1\cup Y_2$ and $\Ind(Y)\cap \F\neq \emptyset$, one has either
$\Ind(Y_1)\cap \F\neq \emptyset$ or $\Ind(Y_2)\cap \F\neq \emptyset$.
The converse holds if furthermore $\F=b\F$.
\end{lem}
\begin{proof} Suppose that $\F$ has the Ramsey property. Consider a
t.d.s. $(X, T)$ and closed subsets $Y, Y_1, Y_2$ of $X$ with
$Y=Y_1\cup Y_2$ and $\Ind(Y)\cap \F\neq \emptyset$. Take $F\in
\Ind(Y)\cap \F$. Then $\cap_{n\in F}T^{-n}(Y)\neq \emptyset$. Say,
$x\in \cap_{n\in F}T^{-n}Y$. Set $F_j=\{n\in F: T^nx\in Y_j\}$ for
$j=1, 2$. Then $F=F_1\cup F_2$, and hence either $F_1\in \F$ or
$F_2\in \F$. Thus either $\Ind(Y_1)\cap \F\neq \emptyset$ or
$\Ind(Y_2)\cap \F\neq \emptyset$.

Now suppose that $\F=b\F$, and for any t.d.s. $(X, T)$ and closed
subsets $Y, Y_1, Y_2$ of $X$ with $Y=Y_1\cup Y_2$ and $\Ind(Y)\cap
\F\neq \emptyset$, one has either $\Ind(Y_1)\cap \F\neq \emptyset$ or
$\Ind(Y_2)\cap \F\neq \emptyset$. Let $F\in \F$ and $F=F_1\cup F_2$
with $F_1\cap F_2=\emptyset$. Denote by $X$ the smallest closed
shift-invariant subset of $\{0, 1, 2\}^{\Z}$ containing
$1_{F_1}+2\cdot 1_{F_2}$. Then $F\in \Ind([1]_X\cup [2]_X)$. By
assumption we have either $\Ind([1]_X)\cap \F\neq \emptyset$ or
$\Ind([2]_X)\cap \F\neq \emptyset$. Without loss of generality let us
assume that $\Ind([1]_X)\cap \F\neq \emptyset$. Say, $F'\in
\Ind([1]_X)\cap \F$. Since $X$ is the orbit closure of
$1_{F_1}+2\cdot 1_{F_2}$, it follows that for any finite subset $W$
of $F'$ there exists some $m\in \Z$ with $m+W\subseteq F_1$. Thus
$F_1\in b\F=\F$. Therefore $\F$ has the Ramsey property.
\end{proof}

From Proposition~\ref{same indep:prop} and Lemma~\ref{dynamical
Ramsey to Ramsey:lemma} we get:

\begin{prop}\label{ye4} Let $\F$ be a family. If $\F$ has the dynamical Ramsey
property, then $b\F$ has the Ramsey property.
\end{prop}

We remark that if $\F$ has the dynamical Ramsey property, it is not
necessarily true that $\F$ has the Ramsey property. For example,
$\F_{\pd}$  has the dynamical Ramsey property, but not the Ramsey property.

It is easy to see that $\F_{\ps}$ has the Ramsey property.
It is also known that $\F_{\cen}$ has the Ramsey property \cite[Corollary 2.16]{Ber03}.
The
celebrated Hindman theorem \cite{Hindman74} says that $\F_{\ip}$ has
the Ramsey property.
This leads to the following questions:

\begin{ques} \label{Ramsey vs dynamical Ramsey:ques}
Is there any family which has the Ramsey property but not the dynamical Ramsey property?
\end{ques}

\begin{ques} \label{dynamical Ramsey for ps and ip:ques}
Do the families $\F_{\ps}$ and $\F_{\ip}$ have the dynamical Ramsey
property?
\end{ques}

To end the section we shall discuss $1$-independence for various
families.
Denote by $\F_{\rs}$ the family generated by  $\{n\Z_+:
n\in\N\}$. The following notion was introduced in \cite{HY3}. Let
$(X,T)$ be a t.d.s.. We say that $(X,T)$ has {\it dense small periodic
sets}, if for any nonempty open subset $U$ of $X$ there exist a
nonempty closed $A\subseteq U$ and $k\in\N$ such that $T^kA\subseteq A$.
To state our result we need a local version of
this notion. That is, for a point $x$ in a t.d.s. $(X,T)$, $x$ is
called {\it quasi regular} if for each neighborhood $U$ of $x$,
there exist a nonempty closed $A\subseteq U$ and $k\in\N$ such that
$T^kA\subseteq A$.
The closed set of quasi regular points of $T$
is denoted by $\QR(T)$.

\begin{thm}\label{singleset}
Let $(X,T)$ be a t.d.s.. Then
\begin{enumerate}
\item $x\in X$ is $\F_{\ip}$-independent iff
$x\in \overline{\Rec(T)}$, where $\Rec(T)$ denotes the set of recurrent
points of $T$. Thus, $(X,T)$ is $\F_{\ip}$-independent of order 1 iff
$\overline{\Rec(T)}=X$.

\item $x\in X$ is $\F_{\inf}$-independent iff
$x\in \overline{\Lambda(T)}$, where $\Lambda(T)=\cup_{x\in X} \omega(x,T)$
and $w(x, T)=\cap_{k\ge 0}\overline{\cup_{n\ge k} \{T^nx\}}$.
Thus, $(X,T)$ is $\F_{\inf}$-independent of order 1 iff
$\overline{\Lambda(T)}=X$,
iff $\overline{\Rec(T)}=X$.

\item $x\in X$ is $\F_{\pubd}$-independent iff
 $x\in \overline{\rM(T)}$, where $\rM(T)={\cup_{\mu\in \rM(X, T)} \supp(\mu)}$
 and $\rM(X, T)$ denotes the set of all invariant Borel probability measures on $X$. Thus,
$(X,T)$ is $\F_{\pubd}$-independent of order 1 iff $\overline{\rM(T)}=X$,
iff there exists a $\mu\in \rM(X,T)$ with full support.

\item $x\in X$ is $\F_{\ps}$-independent if
$x\in \overline{\AP(T)}$, where $\AP(T)$ denotes the set of minimal points
of $T$. Thus, $(X,T)$ is $\F_{\ps}$-independent of order 1 iff
$\overline{\AP(T)}=X$.

\item $x\in X$ is $\F_{\rs}$-independent iff
$x\in \QR(T)$. Thus, $(X,T)$ is $\F_{\rs}$-independent of
order 1 iff $\QR(T)=X$.
\end{enumerate}
\end{thm}
\begin{proof}
(1). Assume that $x\in X$ is $\F_{\ip}$-independent and $U$ is a closed
neighborhood of $x$. Then $\Ind(U)\cap \F_{\ip}\not=\emptyset$, and
hence there are an $\IP$-set $F$ and $y\in X$ such that $T^{i}y\in U$
for each $i\in F$. By \cite[Theorem 5]{BF}, $U\cap
\Rec(T)\not=\emptyset$, i.e. $x\in \overline{\Rec(T)}$.

Conversely, assume that $x\in \overline{\Rec(T)}$ and $U$ is an open
neighborhood of $x$. Then there  exists a $y\in \Rec(T)\cap U$.
By \cite[Theorem 2.17]{HF}, the set $N(y, U)$
contains an $\IP$-set. Thus $\Ind(U)\cap \F_{\ip}\not=\emptyset$.

(2). The first statement follows easily from the
definition. The statement that the $\F_{\inf}$-independence of order $1$
for $(X, T)$ implies
$\overline{\Rec(T)}=X$ follows from
the fact that if $(X, T)$ is {\it non-wandering}
in the sense that $\N\cap N(U, U)\neq \emptyset$ for every nonempty open subset $U$ of $X$,
then $\overline{\Rec(T)}=X$ \cite[Theorem 1.27]{HF}.

(3). This was proved in \cite[Proposition 3.12]{KL2}.

(4). Assume that $x\in X$ is $\F_{\ps}$-independent and $U$ is a closed
neighborhood of $x$. Then $\Ind(U)\cap \F_{\ps}\not=\emptyset$, and
hence there are a piecewise syndetic set $F$ and $y\in X$ such that
$T^{i}y\in U$ for each $i\in F$. By \cite[Theorem 6]{BF}, $U\cap \AP(T)\not=\emptyset$,
i.e. $x\in \overline{\AP(T)}$.

Conversely, assume that $x\in \overline{\AP(T)}$ and $U$ is an open
neighborhood of $x$. Then there is $y\in \AP(T)\cap U$. By a
well-known result of Gottschalk, $N(y, U)$ contains a syndetic
set. Thus $\Ind(U)\cap \F_{\ps}\not=\emptyset.$

(5). It is clear that if $x\in \QR(T)$ then $x$ is an
$\F_{\rs}$-independent point. Assume now that $x$ is an
$\F_{\rs}$-independent point. Let $U$ be a closed  neighborhood of
$x$. Then there is exists a $k\in\N$ such that $k\Z_+$ is in $\Ind(U)$.
Take $z\in \cap_{n\in \Z_+}T^{-kn}U$.
Then $T^{kn}z\in U$ for all $n\in\Z_+$. Thus
$A:=\overline{\{T^{kn}z:n\in\Z_+\}}$ is contained in $U$. It is clear that
$T^kA\subseteq A$.
\end{proof}

\begin{rem} \label{bFrs not Ramsey:cor}
The family $b\F_{\rs}$ does not have the Ramsey property.
\end{rem}
\begin{proof} Let $(X,T)$ be a non-trivial {\it totally minimal} t.d.s., i.e., $X$ is minimal
under $T^k$ for every $k\in \N$. For example, any minimal $(X, T)$ with
$X$ a connected topological space is totally minimal \cite[II(9.6)8]{Vries}.
Let $U$ be a nonempty
open subset of $X$ with $\overline{U}\not=X$. Then $X=X_1\cup X_2$
with $X_1=\overline{U}$ and $X_2=X\setminus U$. Let
$y\in X$. We claim that $N(y,X_i)\not\in  b\F_{\rs}$ for each
$i=1,2$. Assume the contrary that $N(y, X_1)\in b\F_{\rs}$. This
means that there are $d\in \N$ and a sequence $\{n_i\}_{i\in \N}$ in $\Z_+$ such that for
each $i$, $T^{n_i+dj}(y)\in X_1$ for each $0\le j\le i$. Replacing $\{n_i\}_{i\in \N}$ by
a subsequence if necessary, we may assume that $T^{n_i}(y)$ converges to some $z\in X$. Then $z\in X_1$
and $T^{dj}(z)\in X_1$ for each $j\in\N$, contradicting the assumption that
$(X,T)$ is totally minimal. The same argument shows that
$N(y,X_2)\not\in  b\F_{\rs}$. Since $\Z_+=N(y, X)=N(y,X_1)\cup N(y,X_2)$, we
conclude that $b\F_{\rs}$ does not have the Ramsey property.
\end{proof}

\section{Independence: measurable case}

In this section, for a given family $\F$, we define
$\F$-independence for m.d.s., and discuss $1$-independence for
various families. First we define independence sets for m.d.s.,
similar to that for t.d.s. in Definition~\ref{indede}.

\begin{de}\label{indede1}
Let $(X,\B, \mu, T)$ be a m.d.s..
For a tuple
$\bfA=(A_1,\ldots,A_k)$ of sets in $\B$, we say that a subset
$F\subseteq \Z_+$ is an {\it independence set} for $\bfA$ if for
any nonempty finite subset $J\subseteq F$, we have $$\mu(\bigcap_{j\in
J}T^{-j}A_{s(j)})>0$$ for any $s\in \{1,\dots,k\}^J$.
\end{de}

We shall still denote the collection of all independence sets for $\bfA$ by
$\Ind(A_1,\ldots, A_k)$ or $\Ind \bfA$.
Note that Lemma~\ref{ye1}.(1)-(3) holds also
for m.d.s..

\begin{prop} \label{dynamcal Ramsey:prop}
Let $\F$ be a family with the dynamical Ramsey property. For any
m.d.s. $(X, \B, \mu, T)$, any $k\in \N$ and $A_1, A_2, \dots, A_k,
A_{1,1}, A_{1, 2}\in \B$ with $A_1=A_{1, 1}\cup A_{1, 2}$, if
$\Ind(A_1, A_2, \dots, A_k)\cap \F\neq \emptyset$, then either
$\Ind(A_{1, 1}, A_2, \dots, A_k)\cap \F\neq \emptyset$ or
$\Ind(A_{1, 2}, A_2, \dots, A_k)\cap \F\neq \emptyset$.
\end{prop}
\begin{proof}
Set $B_{k+1}=X$, $B_0=A_{1, 1}$, $B_1=A_{1, 2}$, and $B_i=A_i$ for
$2\le i\le k$. Denote by $Y$ the set of elements $s$ in
$\Sigma_{k+2}:=\{0, 1, \dots, k+1\}^{\Z_+}$ satisfying that for any
nonempty finite subset $J$ of $\Z_+$, $\mu(\bigcap_{j\in
J}T^{-j}A_{s(j)})>0$.  Then $Y$ is a closed subset of
$\Sigma_{k+2}$, and contains the constant function $k+1$. It is also
easily checked that $\sigma(Y)=Y$, where $\sigma$ denotes the shift
map. Thus $(Y, \sigma)$ is a t.d.s.. Note that $\Ind(A_1, A_2,
\dots, A_k)=\Ind([0]_Y\cup[1]_Y, [2]_Y, \dots, [k]_Y)$, $\Ind(A_{1,
1}, A_2, \dots, A_k)=\Ind([0]_Y, [2]_Y, \dots, [k]_Y)$, and
$\Ind(A_{1, 2}, A_2, \dots, A_k)=\Ind([1]_Y, [2]_Y, \dots, [k]_Y)$.
Thus $\Ind([0]_Y\cup[1]_Y, [2]_Y, \dots, [k]_Y)\cap \F\neq
\emptyset$. Since $\F$ has the dynamical Ramsey property, either
$\Ind([0]_Y, [2]_Y, \dots, [k]_Y)\cap \F\neq \emptyset$ or
$\Ind([1]_Y, [2]_Y, \dots, [k]_Y)\cap \F\neq \emptyset$. That is,
either $\Ind(A_{1, 1}, A_2, \dots, A_k)\cap \F\neq \emptyset$ or
$\Ind(A_{1, 2}, A_2, \dots, A_k)\cap \F\neq \emptyset$.
\end{proof}

Next we define $\F$-independence for m.d.s., similar to that for t.d.s.
in Definition~\ref{tuple:def}.

\begin{de} \label{tuple1:def}
Let $\F$ be a family and $k\in\N$.
We say that a m.d.s. $(X,\mathcal{B},\mu,T)$ is
{\it $\F$-independent of order $k$} if for each tuple
$(A_1,\ldots, A_k)$ of sets in $\B$ with positive measures,
$\Ind(A_1,\dots, A_n)\cap \F\not=\emptyset$. It is said to be
{\it $\F$-independent}, if it is $\F$-independent of order
$k$ for each $k\in\N$.
\end{de}

Note that Proposition~\ref{same indep:prop}.(1) holds also
for m.d.s..

\begin{rem} \label{measure algebra:rem}
Given a probability space $(X, \B, \mu)$, one may consider the
equivalence relation defined on $\B$ by $A\sim B$ exactly when
$\mu(A\Delta B)=0$, where $A\Delta B=(A\setminus B)\cup (B\setminus
A)$ is the symmetric difference of $A$ and $B$. The set of
equivalence classes in $\B$, denoted by $\tilde{\B}$, has the
induced operation of taking complement and countable union.
Furthermore, $\mu$ descends to a function $\tilde{\mu}$ on
$\tilde{\B}$. The pair $(\tilde{B}, \tilde{\mu})$ is called a {\it
measure algebra} \cite[Section 5.1]{HF} \cite[Section 2.1]{EG}.
Given a measurable and measure-preserving map $T:X\rightarrow X$,
one also gets an induced map $\widetilde{T^{-1}}:
\tilde{\B}\rightarrow \tilde{\B}$ preserving $\tilde{\mu}$,
complement and countable union. For any family $\F$ and $k\in \N$,
it is clear that whether a m.d.s. $(X, \B, \mu, T)$ is
$\F$-independent of order $k$ or not depends only on the triple
$(\tilde{\B}, \tilde{\mu}, \widetilde{T^{-1}})$.
\end{rem}

Consider a m.d.s. $(X, \B, \mu, T)$  or a t.d.s. $(X, T)$. Let
$\bfA=(A_1, \dots, A_n)$ be a tuple of subsets of $X$ (in $\B$ for
m.d.s.). For each $k\in \N$ set $a_k=\max_{F\in \Ind \bfA}|F\cap [0,
k-1]|$. Then the function $k\mapsto a_k$ on $\N$ is {\it
subadditive} in the sense that $a_{k+j}\le a_k+a_j$ for all $k, j\in
\N$. Thus the limit $\lim_{k\to +\infty}\frac{a_k}{k}$ exists and is
equal to $\inf_{k\in \N}\frac{a_k}{k}$ (see for example
\cite[Theorem 4.9]{PW}). We call  this limit the {\it independence
density} of $\bfA$ and denote it by $I(\bfA)$ (see the discussion
before Proposition 3.23 in \cite{KL2} for the case of actions of
discrete amenable groups). The following lemma was proved by Glasner
and Weiss in the second paragraph of the proof of Theorem 3.2 in
\cite{GW95}, using Birkhoff's ergodic theorem. We give a topological
proof here.

\begin{lem} \label{independence density:lemma}
There exists $F\in \Ind \bfA$ with $d(F)=I(\bfA)$.
\end{lem}
\begin{proof} For each $k\in \N$ we claim that there exists $F_k\in \Ind \bfA$ such
that $F_k\subseteq [0, k-1]$ and $|F_k\cap [0, j-1]|\ge j
(I(\bfA)-\frac{1}{k})$ for all $1\le j\le k$. Suppose that this is
not true. Then $I(\bfA)-\frac{1}{k}>0$. Furthermore, for any $F\in
\Ind \bfA$ we can find a strictly increasing sequence $\{b_i\}_{i\in
\N}$ in $\Z_+$ such that $b_1=0$, and $b_{i+1}-b_i\le k$ and $|F\cap
[b_i, b_{i+1})|<(b_{i+1}-b_i) (I(\bfA)-\frac{1}{k})$ for all $i\in
\N$. Set $m=k^2+1$. Take $F\in \Ind \bfA$ with $|F\cap [0,
m-1]|=a_m$, and let $\{b_i\}_{i\in \N}$ be as above. Then $b_s<m\le
b_{s+1}$ for some $s\in \N$. Thus
\begin{eqnarray*}
m\cdot I(\bfA)&\le& a_m=|F\cap [0, m-1]|=|F\cap [b_s, m-1]|
+\sum^{s-1}_{i=1}|F\cap [b_i, b_{i+1})| \\
&\le& k+\sum^{s-1}_{i=1}(b_{i+1}-b_i)(I(\bfA)
-\frac{1}{k})\le k+m(I(\bfA)-\frac{1}{k}),
\end{eqnarray*}
which contradicts $m=k^2+1$. This proves our claim.

Now some subsequence of $\{1_{F_k}\}_{k\in \N}$ converges in $\{0,
1\}^{\Z_+}$ to $1_F$ for some $F\in \Z_+$. Clearly $F\in \Ind \bfA$
and $|F\cap [0, k-1]|\ge k \cdot I(\bfA)$ for every $k\in \N$. We
also have $\limsup_{k\to +\infty}\frac{|F\cap [0, k-1]|}{k}\le
\lim_{k\to +\infty}\frac{a_k}{k}=I(\bfA)$. Therefore $F$ has density
$I(\bfA)$.
\end{proof}

We now discuss 1-independence for various families.
Using Birkhoff's ergodic theorem,
Bergelsen proved part (1) of the following theorem \cite[Theorem 1.2]{Ber}.
Here we
give a different proof.

\begin{thm}\label{ber}
Let $(X, \B, \mu, T)$ be a m.d.s..
The following hold:
\begin{enumerate}
\item For any $A\in \B$ with $\mu(A)>0$, there exists $F\in \F_{\pd}\cap \Ind(A)$ with
density at least $\mu(A)$. In particular, $(X,\mathcal{B},\mu,T)$ is $\F_{\pd}$-independent of
order 1.

\item $(X, \B,\mu,T)$ is $\F_{\rms}$-independent of order $1$ iff $T$ is a.e. periodic,
iff $(X, \B, \mu, T)$ is $\F_{\rs}$-independent of order $1$,  iff
for each $A\in\mathcal{B}$, a.e. every point of $A$
returns to $A$ syndetically, iff for each $A\in\mathcal{B}$,
 a.e. every point of $A$ returns to $A$ along
$n\Z_+$ for some $n\in\N$.

\item Let $(X,\mathcal{B},\mu,T)$ be a m.d.s. If
$A$ is in $\mathcal{B}$ with $\mu(\bigcup_{i=0}^{n-1} T^{-i}A)=1$
for some $n\in \N$, then $\Ind(A)\cap \F_{\rms}\not=\emptyset$.
\end{enumerate}
\end{thm}
\begin{proof}

(1).  For each $k\in \N$ let $a_k$ be defined as before Lemma~\ref{independence density:lemma}
for $\bfA=(A)$. Then $\sum^{k-1}_{j=0}1_{T^{-j}A}\le a_k$ a.e. on $X$. Thus
$k\mu(A)=\int_X\sum^{k-1}_{j=0}1_{T^{-j}A}\, d\mu\le a_k$. It follows that
$I(\bfA)\ge \mu(A)$. By Lemma~\ref{independence density:lemma} we can find
$F\in \Ind(A)$ with $d(F)=I(\bfA)\ge \mu(A)$.

(2). By Theorem~\ref{notsyndetic} the first condition implies the second one.
Clearly the second condition implies the third one and the fifth one,  the third one implies the first one,
and the fifth one implies the fourth one.
Thus it suffices to show that the fourth condition implies the first one.

Let $A\in\mathcal{B}$ with $\mu(A)>0$ and assume that
a.e. every point of $X$ returns to $A$ syndetically.
For each $n\in \N$ set $A_n=\bigcap_{j\in \Z_+}\bigcup^{n-1}_{i=0}T^{-j-i}A$.
Then
$\mu(X\setminus\bigcup_{n\in\N} A_n)=0$ and thus there exists $n\in\N$
with $\mu(A_n)>0$. Denote by $N$ the union of the measure zero ones among $\bigcap_{j\in J}T^{-j}A$
for $J$ running over nonempty finite subsets of $\Z_+$.
Then $\mu(N)=0$, and hence $\mu(A_n\setminus N)=\mu(A_n)>0$.
Take $x\in A_n\setminus N$.
Then there exists
$F\in\F_{\rms}$ such that for each $j\in F$, $T^jx\in A$.
For each nonempty finite subset $J$ of $F$, $x\in (\bigcap_{j\in J}T^{-j}A)\setminus N$,
thus $\mu(\bigcap_{j\in J}T^{-j}A)>0$.
That is, $F\in \Ind(A)$.

(3). The condition implies that $\mu(A_n)=1$. Thus the conclusion follows from the last paragraph.
\end{proof}

\section{Classes of topological $\F$-independence}
\subsection{General discussion}

In this subsection we characterize $\F_{\inf}$ (resp. $\F_{\ip}$)
independent t.d.s. in Theorem~\ref{weakmixing}, construct a
nontrivial topological K system with a unique minimal point in
Example~\ref{topological K unique minimal:example}, and discuss
$\F_{\rs}$-independence at the end.

A t.d.s. $(X, T)$ is said to be {\it (topologically) transitive} if for any nonempty open
subsets $U$ and $V$ of $X$, $N(U, V)$ is nonempty; it is called {\it weakly mixing} if
$(X\times X, T\times T)$ is transitive.
The equivalence of the conditions (1), (2) and (3) in the following theorem was proved
in \cite[Theorem 8.6]{KL2}.
Here we strengthen it by adding the conditions (4) and (5).

\begin{thm} \label{weakmixing}
For a t.d.s. $(X,T)$ the following are equivalent:
\begin{enumerate}
\item $(X,T)$ is weakly mixing.

\item $(X, T)$ is $\F_{\inf}$-independent of order $2$.

\item $(X,T)$ is $\F_{\inf}$-independent.

\item $(X, T)$ is $\F_{\ip}$-independent of order $2$.

\item $(X,T)$ is $\F_{\ip}$-independent.
\end{enumerate}
\end{thm}
\begin{proof} It is clear that (5)$\Rightarrow$(4)$\Rightarrow$(2) and
(5)$\Rightarrow$(3)$\Rightarrow$(2). The implication (2)$\Rightarrow$(1)
follows from the fact that if for any nonempty open subsets $U$ and $V$ of $X$
one has $N(U, U)\cap N(U, V)\neq \emptyset$ , then $(X, T)$ is weakly mixing \cite[Lemma]{Pet70}.
(\cite[Lemma]{Pet70} was proved only for invertible t.d.s., but it is easy to modify the proof to make it
work for any t.d.s..)
Thus it suffices to
show that (1)$\Rightarrow$(5).

Now assume that $(X, T)$
is weakly mixing. Then each $(X\times \dots \times X, T\times \dots \times T)$ is
transitive \cite[Proposition II.3]{F}. Thus, for any $n\in\N$, if $U_1,\ldots, U_n$ and
$V_1,\ldots, V_n$ are nonempty open subsets of $X$, then
$\N\cap (\bigcap_{i=1}^n N(U_i,V_i))\not=\emptyset$. For any given nonempty
open subsets $U_1,\ldots, U_n$ of $X$, we are going to find an $\IP$-set $F$
in $\Ind(U_1,\ldots, U_n)$.

First there exists a $t_1\in \N$ such that $t_1\in \bigcap_{(i_1,i_2)\in
\{1,\ldots,n\}^2} N(U_{i_1},U_{i_2})$. Assume that $t_1,\dots, t_k$ in $\N$
are defined such that $\{a_1,\ldots,a_j\}:=\{0\}\cup \{t_{i_1}+\dots
+t_{i_l}: 1\le i_1<\dots<i_l\le k\}$ is in $\Ind(U_1,\ldots, U_n)$. Pick
$t_{k+1}\in \N$ such that
$$t_{k+1}\in \bigcap_{1\le i_m\le n, 1\le m\le 2j}
N(T^{-a_1}U_{i_1}\cap\ldots \cap T^{-a_j}U_{i_j},
T^{-a_1}U_{i_{j+1}}\cap\ldots \cap T^{-a_j}U_{i_{2j}}).$$ Then
$\{0\}\cup \{t_{i_1}+\ldots +t_{i_l}: 1\le i_1<\ldots<i_l\le
k+1\}=\{a_1, \dots, a_j\}\cup \{a_1+t_{k+1}, \dots, a_j+t_{k+1}\}$ is in
$\Ind(U_1,\dots, U_n)$. This implies that the IP-set
generated by the sequence $\{t_i\}_{i\in \N}$ is in $\Ind(U_1,\dots, U_n)$.
\end{proof}

Petersen \cite{Petersen} showed there exists a t.d.s. which is
strictly ergodic, strongly mixing, and has zero topological entropy.
Thus in such a system every tuple is $\F_{\ip}$-independent, while
no non-diagonal tuple is $\F_{\pd}$-independent.

A t.d.s. is called an {\it $E$-system} if it is transitive and
has an invariant Borel probability measure with full support;
it is called an {\it $M$-system} if it is transitive and the set of minimal
points is dense; it is called {\it totally transitive} if
$(X, T^k)$ is transitive for every $k\in \N$.
By Theorems~\ref{singleset} and \ref{weakmixing} we have

\begin{cor} Let $(X,T)$ be a t.d.s.. The following hold:
\begin{enumerate}
\item  If $(X,T)$ is $\F_{\pd}$-independent of order $2$, then it is an $E$-system.

\item If $(X,T)$ is $\F_{\ps}$-independent of order $2$,  then it is an $M$-system.

\item If $(X,T)$ is $\F_{\rs}$-independent of order $1$, then it has dense small
periodic sets. If it is $\F_{\rs}$-independent of order $2$,
then it is totally transitive,
and hence is disjoint from all minimal systems by
\cite[Theorem 3.4]{HY3}.
\end{enumerate}
\end{cor}

\begin{de} \label{topological K:def}
We say that a t.d.s. is {\it topological K} if it is $\F_{\pubd}$-independent.
\end{de}

By \cite[Theorem 8.3]{HY2} and \cite[Theorem 3.16]{KL2} a t.d.s.  is topological K
if and only if each of its finite covers by non-dense open subsets has positive topological entropy.

Next we show that there is an invertible  topological K system with only
one minimal point.
Recall that a t.d.s. $(X, T)$ is said to be {\it proximal} if the orbit closure of
every point in $(X\times X, T\times T)$ has nonempty intersection with the diagonal.
Following \cite{KL2} we shall refer to $\F_{\pubd}$-independent tuples of a t.d.s. as $\IE$-tuples.
To construct the example we need

\begin{lem} \label{unique minimal:lemma}
Let $(X, T)$ be a t.d.s.. We have:
\begin{enumerate}
\item Suppose that $(X, T)$ has  a transitive point
$x$. Then $T$ is topologically K if and only if for each $j\in\N$,
$(x, Tx, \dots, T^{j-1}x)$ is an $\IE$-tuple.

\item $(X,T)$ has only
one minimal point if and only if $(X,T)$ is proximal.
\end{enumerate}
\end{lem}
\begin{proof}
(1). This follows from the fact that the set of $\IE$ $j$-tuples is
closed in $X^j$ for each $j\in \N$.

(2). The ``only if" part  is trivial. Assume that $(X, T)$ is proximal.
Take $x\in X$. Say, $(y, y)$ is in the intersection of the diagonal and the orbit closure of $(x, Tx)$.
Then $Ty=y$. Let $z\in X$. Then the orbit closures of $y$ and $z$ have nonempty intersection,
which of course has to be $\{y\}$. It follows that if $z$ is minimal, then $z=y$.
\end{proof}

For a t.d.s. $(X, T)$, recall its {\it natural extension} $(\tilde{X}, \tilde{T})$
defined as follows. $\tilde{X}$ is the closed subspace of $\prod_{n\in \N}X$ consisting
of $(x_1, x_2, \dots)$ with $T(x_{n+1})=x_n$  for all $n\in \N$, and
$\tilde{T}$ is defined as $\tilde{T}(x_1, x_2, \dots)=(T(x_1), x_1, x_2, \dots)$.
Note that $\tilde{T}$ is a homeomorphism and the projection $\pi: \tilde{X}\rightarrow X$ sending
$(x_1, x_2, \dots)$ to $x_1$ is a factor map.
It is well known that $(X, T)$ and $(\tilde{X}, \tilde{T})$ share many
dynamical properties. Here we need a special case.

\begin{lem} \label{natural extension:lem}
Let $(X, T)$ be a t.d.s.. The following are true:
\begin{enumerate}
\item Let $\F$ be a family and $k\in \N$. Then $(X, T)$ is $\F$-independent of
order $k$ if and only if $(\tilde{X}, \tilde{T})$ is so.

\item $(X, T)$ is proximal if and only if $(\tilde{X}, \tilde{T})$ is so.
\end{enumerate}
\end{lem}
\begin{proof}
(1). The ``if" part follows from the fact that if a t.d.s. is
$\F$-independent of order $k$, then so is every factor. Suppose
that $(X, T)$ is $\F$-independent of order $k$. Let $U_1, \dots,
U_k$ be nonempty open subsets of $\tilde{X}$. Then there exist
nonempty open subsets $V_1, \dots, V_k$ of $X$ and $m\in \N$ such
that if $(x_1, x_2, \dots)$ is in $\tilde{X}$ and $x_m\in V_j$ for
some $1\le j\le k$, then $(x_1, x_2, \dots)$ is in $U_j$.

We claim that $\Ind(V_1, \dots, V_k)\subseteq \Ind(U_1, \dots, U_k)$.
Let $F\in \Ind(V_1, \dots, V_k)$, $J$ be a nonempty finite subset of $F$, and
$s\in \{1, \dots, k\}^J$. Then $\bigcap_{j\in J}T^{-j}V_{s(j)}\neq \emptyset$.
Take $y\in  \bigcap_{j\in J}T^{-j}V_{s(j)}$. We can find $\tilde{x}=(x_1, x_2, \dots)\in
\tilde{X}$  such that $x_m=y$. Then $\tilde{x}\in \bigcap_{j\in J}\tilde{T}^{-j}U_{s(j)}$.
Thus $F\in \Ind(U_1, \dots, U_k)$. This proves our claim.

Since $\F\cap \Ind(V_1, \dots, V_k)\neq \emptyset$, we get $\F\cap \Ind(U_1, \dots, U_k)\neq \emptyset$.
Therefore $(\tilde{X}, \tilde{T})$ is also
$\F$-independent of order $k$.
This proves the ``only if" part.

(2). This is trivial.
\end{proof}

For $p\ge 2$ let $\Lambda_p=\{0,1,\ldots,p-1\}$ with the discrete
topology, $\Sigma_p=\Lambda_p^{\Z_+}$ with the product topology and
$\sigma:\Sigma_p\rightarrow \Sigma_p$ be the shift. For $n\in \N$ and
$a=(a(1),a(2),\ldots, a(n))\in \Lambda^n_p$ (a block of length $n$),
let $|a|=n$, $\sigma(a)=(a(2),\ldots,a(n))$. We say that $a$ {\it
appears} in $x=(x(1),x(2),\ldots)\in\Sigma_p$ or $x\in \Lambda^m_p$
with $m\ge n$ if there is $j\in\N$ with $a=(x(j),x({j+1}),\ldots,
x({j+n-1}))$ (write $a<x$ for short) and we use $t^i$ to denote
${t\ldots t}$ ($i$ times). For $b=(b(1),\ldots,b(m))\in
\Lambda^m_p$, denote $(a(1),\ldots,a(n),b(1),\ldots,b(m))\in
\Lambda^{n+m}_p$ by $ab$. Denote $(ii\ldots)$ by $\bf i$, $0\le i\le p-1$.
We also need the
following lemma.  In view of \cite[Proposition 2]{B2} and
\cite[Theorem 7.3]{HY2} or \cite[Theorem 3.16]{KL2}, it is equivalent to
 \cite[Lemma 4.1]{HKY}. One can also prove it directly using $\IE$-pairs instead
 of entropy pairs in the proof of \cite[Lemma 4.1]{HKY}.

\begin{lem} \label{E unique minimal:lem}
There is an $E$-system $(Y,\sigma)$ contained in
$(\Sigma_3,\sigma)$ with a unique minimal point $\bf 0$ such that
$Y$ has  an $\IE$-pair $(x_1, x_2)$ in $[1]_Y\times [2]_Y$.
\end{lem}

\begin{ex} \label{topological K unique minimal:example}
There exists a non-trivial invertible t.d.s. which is topological K
and has a unique minimal point.
\end{ex}
\begin{proof}
By Lemmas~\ref{unique minimal:lemma} and \ref{natural extension:lem} it suffices to show that there
exists a non-trivial t.d.s. which is topological K and proximal. We use
the idea in  the proof of Theorem 4.2 in \cite{HKY}. The main idea
is to construct a recurrent point $x\in \Sigma_2$ with the following
two properties:

\begin{itemize}

\item[{(I)}] for any $j\in \N$, $(x, \sigma(x), \dots, \sigma^{j-1}(x))$ is an
$\IE$-tuple of $(X, \sigma)$, where $X$ is the orbit closure of $x$, and

\item[{(II)}] for each $n\in \N$, $0^n$ appears in $x$
syndetically.
\end{itemize}

By Lemma~\ref{unique minimal:lemma} it is clear that $(X,\sigma)$ is
topological K
with a unique minimal point $\bf 0$.
First we give the detailed construction of the recurrent point $x$.

Let $(Y,\sigma)$ be the system constructed in Lemma~\ref{E unique minimal:lem} and
let $y$ be a transitive point of $Y$. By \cite[Theorem 3.18]{KL2}
for each $m\in \N$ we can find an $\IE$-tuple $(z_{m, 1}, \dots,
z_{m, m})$ of $Y$ with $z_{m, 1}, \dots, z_{m, m}$ pairwise
distinct and all in $[1]_Y\cup [2]_Y$. Then  we can find $t_m\in \N$
such that $z_{m, 1}[0, t_m], \dots, z_{m, m}[0, t_m]$ are
pairwise distinct, where $z[0, t]$ denotes $(z(0), \dots, z(t))$.
Define a map $f_m: \Lambda_3^{t_m+1}\rightarrow \Lambda_{m+1}$ by
$f_m(a)=j$ if $a=z_{m, j}[0, t_m]$ for some $1\le j\le m$ and
$f_m(a)=0$ otherwise.

Take $\phi:\N\rightarrow \N$ such that $\phi(k)\le k$ for each $k\in
\N$ and for each $m\in\N$, $\phi^{-1}(m)$ is infinite.

Set $A_1=(10)$, $n_1=|A_1|=2$. Set
$$C_{1,0}=0^{2n_1}, \; C_{1,1}=A_10^{n_1}.$$

Suppose that  $A_1,\ldots, A_k$, $C_{m,i}$ for $0<m\le k$ and $0\le i\le m$, and $n_1,\dots,
n_k$ are defined. We define inductively $A_{k+1}$, $C_{k+1,i}$ for
$i=0, 1,\dots,k+1$, and $n_{k+1}$.

Say, $m=\phi(k)$. Since $(Y,\sigma)$ has a unique minimal point $\bf
0$, there exists $\ell_k\in \N$ with $\ell_k\ge t_m$ such that
$0^{n_k}$ appears in $y$ with gaps bounded above by $\ell_k$. Set
$b_k=2\ell_k n_k$, and set
\begin{align*}
&A_{k+1}=A_k0^{n_k}C_{m, f_m(y[0, t_m])}C_{m, f_m(y[1,
t_m+1])}\ldots
C_{m, f_m(y[b_k-t_m, b_k])}0^{2n_k},\; n_{k+1}=|A_{k+1}|,\ \text{ and}\\
&C_{k+1, 0}=0^{2n_{k+1}}, C_{k+1,
i}=\sigma^{i-1}(A_{k+1})0^{i-1}0^{n_{k+1}} \text{ for
}i=1,2,\dots,k+1.
\end{align*}

It is clear that $x:=\lim_{k\to +\infty}A_k$ is a recurrent point of $\sigma$ in $\Sigma_2$.
Denote by $X$ the orbit closure of $x$ in $\Sigma_2$.
We claim that $x$ satisfies (I)
and (II).

\noindent {\bf (I).} Given $j\in \N$, we show that $(x, \sigma(x),
\dots, \sigma^{j-1}(x))$ is an $\IE$-tuple of $(X, \sigma)$. Suppose
that $V_0', V_1', \dots, V_{j-1}'$ are neighborhoods of $x,
\sigma(x), \dots, \sigma^{j-1}(x)$ respectively. Then there is some
$m\in \N$ with $m>j$ such that $V_i\subseteq V_i'$  for all $0\le i\le
j-1$, where $V_i:=[\sigma^i(A_m)0^i]_X$ for all $0\le i\le m-1$.

Since $(z_{m, 1}, \dots, z_{m, m})$  is an $\IE$-tuple of $Y$, there
exists some $d>0$ such that for any $n\in \N$ we can find a finite
subset $J\subseteq \Z_+$ with $|J|\ge n$ contained in an interval with
length at most $d|J|$
such that for any $s\in \{1, 2, \dots, m\}^J$ one has $\bigcap_{i\in
J}\sigma^{-i}U_{s(i)}\neq \emptyset$, where $U_j=(z_{m, j}[0,
t_m])_Y$ for $1\le j\le m$. Since $y$ is a transitive point of $Y$,
we have $\sigma^N(y)\in \bigcap_{i\in J}\sigma^{-i}U_{s(i)}$ for some
$N\in \Z_+$. Then $y_{[N+i, N+i+t_m]}=z_{m, s(i)}[0, t_m]$ for
all $i\in J$. Take $k\ge N+\max J+t_m$ with $\phi(k)=m$.
Then $b_k\ge k\ge N+i+t_m$ for all $i\in J$, and
$$A_{k+1}=A_k0^{n_k}C_{m, f_m(y[0, t_m])}C_{m, f_m(y[1, t_m+1])}\ldots
C_{m, f_m(y[b_k-t_m, b_k])}0^{2n_k}.$$ Note that $f_m(y[N+i,
N+i+t_m])=s(i)$ for all $i\in J$. Thus
$\sigma^{2n_k+2(N+i)n_m}(x)\in [C_{m, s(i)}]_X\subseteq V_{s(i)-1}$
for all $i\in J$. It follows that $\sigma^{2n_k+2Nn_m}(x)\in
\bigcap_{i\in 2n_mJ}\sigma^{-i}V_{\psi(i)}$ for the map
$\psi\in \{0,
1, \dots, m-1\}^{2n_mJ}$ defined
by $\psi(2n_mi)=s(i)-1$ for all $i\in J$. Therefore $2n_mJ$ is an independence set for
$(V_0, \dots, V_{m-1})$. Clearly $2n_mJ$ is contained in an interval
with length at most $2n_md|J|=2n_md|2n_mJ|$. Thus by Lemma~\ref{independence density:lemma} $(x,
\sigma(x), \dots, \sigma^{j-1}(x))$ is an $\IE$-tuple of $(X,
\sigma)$.

\noindent {\bf (II).} We now show that for each $n\in \mathbb{N}$,
$0^n$ appears in $x$ syndetically. It suffices to prove that for
each $k\in \mathbb{N}$, $0^{n_k}$ appears in $x$ syndetically with
gaps bounded above by $2b_k$.

Fix $k\in \N$. Say, $\phi(k)=m$. By the construction
$$A_{k+1}=A_k0^{n_k}C_{m, f_m(y[0, t_m])}C_{m, f_m(y[1, t_m+1])}\ldots
C_{m, f_m(y[b_k-t_m, b_k])}0^{2n_k}.$$

Note that $f_m(a)=0$ for every $a\in \Lambda_3^{t_m+1}$ with
$a(0)=0$. As $0^{n_k}$ appears in $y$ with gaps bounded above by $\ell_k$,
$0^{n_k}$ appears in $C_{m, f_m(y[0, t_m])}C_{m, f_m(y[1,
t_m+1])}\ldots C_{m, f_m(y[b_k-t_m, b_k])}$ with gaps bounded above by
$2n_m\ell_k\le 2n_k\ell_k=b_k$. Thus $0^{n_k}$ appears in $A_{k+1}$
with gaps bounded above by $b_k+n_k\le 2b_k$.

Assume that $0^{n_k}$ appears in $A_{\ell}$ with gaps bounded above by
$2b_k$, where $\ell\ge k+1$. Now we are going to prove that this is
also true for $\ell+1$. Set $m'=\phi(\ell)$. First note that
$$A_{\ell+1}=A_{\ell}0^{n_{\ell}}C_{m', f_{m'}(y[0, t_{m'}])}C_{m', f_{m'}(y[1, t_{m'}+1])}\ldots
C_{m', f_{m'}(y[b_{\ell}-t_{m'}, b_{\ell}])}0^{2n_{\ell}}.$$

If $m'\ge k+1$, then by the induction assumption and the
construction of $C_{m', i}$ we know that $0^{n_k}$ appears in
$A_{\ell+1}$ with gaps bounded above by $2b_k$. If $m'\le k$, then by the
induction assumption and the discussion similar to the case of
$A_{k+1}$, we know that $0^{n_k}$ appears in $A_{\ell+1}$ with gaps
bounded above by $2b_k$. Hence $0^{n_k}$ appears in $x$ syndetically with
gaps bounded above by $2b_k$,  as $x=\lim_{\ell\to +\infty} A_{\ell}$.
\end{proof}

\begin{de} We say that a t.d.s $(X,T)$ is {\it Bernoulli} if it is conjugate to
$(A^{\Z_+}, \sigma)$, where $A$ is a compact metrizable space with
$|A|\ge 2$ and $\sigma$ is the shift.
\end{de}

\begin{thm} A Bernoulli system  is
$\F_{\rs}$-independent.
\end{thm}
\begin{proof} Let $(X,T)$ be a Bernoulli system. Without loss of generality we may assume
that $(X, T)=(A^{\Z_+}, \sigma)$ as above.
Let $U_1, \dots, U_n$ be nonempty open subsets of $X$ for some $n\in \N$.
Then there exist some $k\in \N$ and nonempty subsets
$A_{i, j}\subseteq A$ for $1\le i\le n$ and $0\le j<k$ such that
$U_i\supseteq A_{i,0}\times \ldots \times A_{i, k-1}\times
\prod_{\ell\ge k} A$ for all $1\le i\le n$. It follows that $k\N\subseteq
\Ind(U_1,\ldots, U_n)$. Thus $(X,T)$ is $\F_{\rs}$-independent.
\end{proof}

Recall that a t.d.s. $(X, T)$ is called {\it strongly mixing} if for any nonempty open
subsets $U$ and $V$ of $X$, $N(U, V)$ is a cofinite subset of $\Z_+$.
In \cite[Example 5]{B1} Blanchard constructed examples of invertible t.d.s. which
are $\F_{\rs}$-independent of order $2$ and are not strongly mixing. In fact,
the Property P defined in \cite{B1} is exactly the same as $\F_{\rs}$-independence of
order $2$. It is easily checked that the condition in \cite[Proposition 4]{B1} actually implies
$\F_{\rs}$-independence. Thus Blanchard's examples are actually $\F_{\rs}$-independent.
Thus $\F_{\rs}$-independence does not imply
strong mixing and hence does not imply Bernoulli.

A factor map $\pi: (X, T)\rightarrow (Y, S)$ between t.d.s. is said
to be an {\it almost one-to-one extension} if the set $\{x\in X:
\pi^{-1}(\pi(x))=\{x\}\}$ is dense in $X$.

For a sequence $K=\{k_n\}_{n\in \N}$ in $\N$ with $k_{n+1}$ being divisible
by $k_n$ for each $n\in \N$, the {\it adding machine} $(X_K, T_K)$
associated to $K$ is defined as follows. $X_K$ is the projective
limit of $\underset{n\to +\infty}\varprojlim \Z/k_n\Z$, as a
metrizable compact abelian group, and $T_K$ is the addition by $1$.

For a t.d.s. $(X,T)$, recall that  $x\in X$ is called a {\it regular
minimal point} \cite[Definition 3.38]{GH} if for each neighborhood
$U$ of $x$, there exists $k\in\N$ such that $N(x,U)\supseteq k\Z_+$.
It is known that if $x$ is a regular minimal point, then its orbit
closure is an almost one-to-one extension of some adding machine,
see for instance \cite[Proposition 3.5]{HY3}. Now we show

\begin{prop}\label{regular} Let $(X,T)$ be a minimal t.d.s.. The following
 are equivalent:

\begin{enumerate}
\item $(X,T)$ has dense small periodic sets.

\item $(X,T)$ is an almost one-to-one extension of some adding machine.

\item $X$ has  a regular minimal point.
\end{enumerate}
\end{prop}
\begin{proof}
By \cite[Proposition 3.5]{HY3} (2) and (3) are equivalent.
(3)$\Rightarrow$(1) is trivial.

(1)$\Rightarrow$(2). For any nonempty open subset $U$ of $X$, let
$B$ be a nonempty closed subset of $U$ with $T^kB\subseteq B$ for some
$k\in \N$. Take $x\in B$. Then the argument in the proof of
\cite[Proposition 3.5]{HY3} shows that the orbit closure  $A$ of $x$ under $T^k$
is a nonempty clopen subset of $U$ and there exists some
$\ell\in \N$ such that $\{A, TA, \dots, T^{\ell-1}A\}$ is a clopen
partition of $X$ and $T^{\ell}A=A$.

Fix a compatible metric on $X$.
Starting with some nonempty open subset $U$ of $X$ with $\diam(U)<1$,
we obtain $A$ and $\ell$ as above, and set $A_1=A$ and
$\ell_1=\ell$. Inductively, assuming that we have found subsets
$A_1\supseteq A_2\supseteq \dots\supseteq A_k$ and positive integers $\ell_1,
\ell_2, \dots, \ell_k$ such that $\diam(A_j)<1/j$, $\{A_j, TA_j,
\dots, T^{\ell_j-1}A_j\}$ is a clopen partition of $X$, and
$T^{\ell_j}A_j=A_j$ for all $1\le j\le k$. We shall find $A_{k+1}$
and $\ell_{k+1}$ with the same property. Let $U$ be a nonempty open
subset of $A_k$ with $\diam(U)<1/(k+1)$. We obtain $A$ and $\ell$ as
above, and set $A_{k+1}=A$ and $\ell_{k+1}=\ell$.

Now the argument in the proof of  \cite[Proposition 3.5]{HY3} shows
that $(X,T)$ is an almost one-to-one extension of some adding
machine.
\end{proof}

\subsection{Non-existence of non-trivial minimal $\F_{\rms}$-independent t.d.s.}

It was shown in \cite[Theorem 3.4]{HY2} that there exist non-trivial minimal
topological K systems (the existence of nontrivial
minimal u.p.e. systems was proved earlier by Glasner and Weiss \cite{GW1}, answering a
question of Blanchard \cite{B2}). As a contrast, we have

\begin{thm} \label{no minimal syndetic tds:thm}
There is no non-trivial minimal t.d.s. which is $\F_{\rms}$-independent
of order $2$.
\end{thm}

To prove Theorem~\ref{no minimal syndetic tds:thm}, we need some
preparation. Crucial to the proof of Theorem~\ref{no minimal
syndetic tds:thm} is the following combinatorial result, which is
also of independent interest. We postpone its proof to the
Appendix. Recall the notion introduced before Lemma~\ref{E unique
minimal:lem}.

\begin{thm} \label{infinite words:lemma}
Let $p, \ell\in \N$ with $p\ge 2$.
For any integer $m\ge 4\ell+2$, given any sequence $\{A_n\}_{n\in \Z_+}$ of subsets
of $\Lambda_p^m$ with $|A_n|\le \ell$\
for each $n\in \Z_+$, there exists
$x\in \Sigma_p$ such that $x[n, n+m-1]\not \in A_n$ for every $n\in \Z_+$.
\end{thm}

We remark that under the conditions of Theorem~\ref{infinite
words:lemma}, the set $\{x\in \Sigma_p: x[n, n+m-1]\not \in A_n
\mbox{ for all } n\in \Z_+\}$ is small in both the topological and
measure-theoretical senses: it is a closed subset of $\Sigma_p$
with empty interior and has measure $0$ for the product measure on
$\Sigma_p$ associated to any probability vector $(t_0, \dots,
t_{p-1})$ with $\sum_{j=0}^{p-1}t_j=1$ and $t_j>0$ for all $0\le
j\le k-1$. The following lemma is important for the proof of
Theorem \ref{no minimal syndetic tds:thm} and also can be applied
to show that an $\F_{\rms}$-independent t.d.s. is disjoint from all
minimal t.d.s. \cite{DSY}.

\begin{lem} \label{no minimal syndetic tds:lemma}
For every minimal subshift $X\subseteq \Sigma_2$, $\Ind([0]_X,[1]_X)$ does not contain any syndetic set.
\end{lem}
\begin{proof}
We argue by contradiction. Assume that $X\subseteq \Sigma_2$ is a
minimal subshift and $\Ind([0]_X,[1]_X)$ contains a syndetic set
$F$. Say, $F=\{n_0<n_1<\dots\}$ with $\ell=\max_{j\in
\Z_+}(n_{j+1}-n_j)$. Let $m$ be as in Theorem~\ref{infinite
words:lemma} for $p=2$ and $\ell$. Take $a\in \Lambda_2^{m\ell}$
such that $a$ appears in some element of $X$. For each $j\in
\Z_+$, set $A_j$ to be the subset of $\Lambda_2^m$ consisting of
elements of the form $(a(k), a(k+n_{j+1}-n_j), a(k+n_{j+2}-n_j),
\dots, a(k+n_{j+m-1}-n_j))$ for $1\le k\le \ell$. Then $|A_j|\le
\ell$ for all $j\in \Z_+$. By Theorem~\ref{infinite words:lemma}
we can find $x\in \Sigma_2$ such that $x[j, j+m-1]\not \in A_j$
for every $j\in \Z_+$. Since $F\in \Ind([0]_X, [1]_X)$, we can
find $y\in X$ with $y(n_j)=x(j)$ for all $j\in \Z_+$. As $X$ is
minimal, there exists some $i\ge n_1$ such that $y[i,
i+m\ell-1]=a$. Say, $n_{j-1}<i\le n_{j}$. Set $k=n_{j}-i+1$. Then
$x(s)=y(n_s)=a(k+n_{s}-n_{j})$ for all $j\le s\le j+m-1$, which
contradicts that $x[j, j+m-1]\not \in A_j$.
\end{proof}

We are ready to prove Theorem~\ref{no minimal syndetic tds:thm}.

\begin{proof}[Proof of Theorem~\ref{no minimal syndetic tds:thm}]
We shall show that if $(Y, S)$ is a minimal t.d.s., and $V_0, V_1$ are disjoint closed subsets of $X$ with
nonempty interior, then $\Ind(V_0, V_1)$ does not contain any syndetic set.

It is well known that we can find a minimal t.d.s. $(X_1, T_1)$
and  a factor map $\pi: (X_1, T_1)\rightarrow (Y,S)$ such that
$X_1$ is a closed subset of a Cantor set (see for eample
\cite[page 34]{BGH}). It is easy to see that $\Ind(V_0,
V_1)=\Ind(\pi^{-1}(V_0), \pi^{-1}(V_1))$. Write $X$ as the
disjoint union of clopen subsets $U_0$ and $U_1$ such that
$U_j\supseteq \pi^{-1}(V_j)$ for $j=0, 1$. Then $\Ind(V_0,
V_1)\subseteq \Ind(U_0, U_1)$.

Define a coding $\phi:X_1\ra \Sigma_2$ such that for each
$x\in X_1$, $\phi(x)=(x_0,x_1,\ldots)$, where $x_i=j$ if
$T_1^i(x)\in U_j$ for all $i\in\Z_+$. Then $X=\phi(X_1)$ is a
minimal subshift contained in $\Sigma_2$ and $\phi:X_1\rightarrow X$ is a factor map. It is easy to verify that
$\Ind(U_0, U_1)\subseteq \Ind([0]_X,[1]_X)$.

By Lemma~\ref{no minimal syndetic tds:lemma} we know that
$\Ind([0]_X, [1]_X)$ does not contain any syndetic set. Then
$\Ind(V_0, V_1)$ does not contain any syndetic set either.
\end{proof}


\subsection{Finite product}

In this subsection we investigate the question for which families
$\F$ the product of finitely many $\F$-independent t.d.s. remains
$\F$-independent.

It is known that if $\F=\F_{\pd}$ the question has a positive
answer \cite[Theorem 8.1]{HY2}
\cite[Theorem 3.15]{KL2}.
We now show that the question has a positive answer for
$\F=\F_{\rs},\F_{\ps}$. It is clear that
$$\F_{\rs}\subseteq \F_{\cen}\subseteq \F_{\ps} \subseteq \F_{\pubd}.$$

We need the following lemma. It is also needed for the proof of
Theorem~\ref{topological K to strongly mixing:thm} later.

\begin{lem} \label{inf-B}
For any $d>0$, $k\in \N$, and finite subset $F\subseteq \Z_+$ with
$d|F|> k$, there exists $N=N(d,k,F)\in \N$ such that for any
nonempty finite interval $I\subseteq \Z_+$ and $S\subseteq I$ with
$\frac{|S|}{|I|}\ge d$ and $|I|\ge N$ one has $|S\cap (F+p)|\ge k$
for some $p\in \Z$.
\end{lem}
\begin{proof} Take $N\in \N$ such that $\frac{d|F|}{1+(\max F)/N}\ge k$.
For each $j\in F$ the set $S-j$ is contained in  $[\min I-\max F,
\max I]\subseteq \Z$. Then we can find some $p\in  [\min I-\max F,
\max I]$ such that $p$ is contained in $S-j$ for at least
$\frac{|S|\cdot |F|}{|[\min I-\max F, \max I]|}$ $j$'s in $F$. Set
$W=\{j\in F: p\in S-j\}$. Then $(W+p)\subseteq S\cap (F+p)$ and
$$ |W|\ge    \frac{|S|\cdot |F|}{|[\min I-\max F,
\max I]|}=\frac{|S|\cdot |F|}{|I|+\max F}\ge \frac{d|F|}{1+(\max
F)/|I|}\ge k .$$
\end{proof}

We need the following simple lemma. For a subset
 $K$ of $\Z_+$, denote by $X_K$ the set of limit points of the sequence
 $\{\sigma^n1_K\}_{n\in \Z_+}$ in $\{0, 1\}^{\Z_+}$, where $\sigma$ denotes the shift map
 on $\{0, 1\}^{\Z_+}$. Note that $(X_K, \sigma)$ is a t.d.s..

\begin{lem}\label{productminimal} The following statements hold:
\begin{enumerate}

\item Let $S_1, S_2\in \F_{\pubd}$. Then there are two
subsets $K_1, K_2$ of $\Z_+$
such that
$1_{K_i}\in X_{S_i}$, $i=1, 2$,
and $K_1\cap K_2\in
\F_{\pud}$.

\item Let $S_1, S_2\in \F_{\ps}$. Then there are
two subsets $K_1, K_2$
of $\Z_+$
such that
$1_{K_i}\in X_{S_i}$, $i=1, 2$,
and $K_1\cap K_2\in
\F_{\rms}\cap \F_{\cen}$.

\item Let $S_1, S_2\in \F_{\rs}$. Then there are two
subsets $K_1, K_2$ of $\Z_+$
such that
$1_{K_i}\in X_{S_i}$, $i=1, 2$,
and $K_1\cap K_2\in
\F_{\rs}$.
\end{enumerate}
\end{lem}
\begin{proof}
(1). Set $X_i=X_{S_i}$. Recall the independence density defined
before Lemma~\ref{independence density:lemma}. We have
$I([1]_{X_i})=BD^*(S_i)>0$ for $i=1, 2$. For each $k\in \N$, take
a finite interval $J_1$ in $\Z_+$ with $|J_1|=k$ and a set $F_1\in
\Ind([1]_{X_1})$ with $F_1\subseteq J_1$ and $|F_1|\ge |J_1|
I([1]_{X_1})$.  Note that we can find arbitrarily long finite
interval $J_2$ in $\Z_+$ and a set $F_2\in \Ind([1]_{X_1})$ with
$F_2\subseteq J_2$ and $|F_2|\ge |J_2| I([1]_{X_2})$. By
Lemma~\ref{inf-B}, when $|I_2|$ is large enough, we have $|F_1\cap
(F_2+p)|\ge I([1]_{X_2})|F_1|-1\ge
|J_1|I([1]_{X_1})I([1]_{X_2})-1$ for some $p\in \Z$. Consider the
t.d.s. $(X_1\times X_2, \sigma\times \sigma)$. Note that $F_1\cap
(F_2+p)\in \Ind([1]_{X_1}\times [1]_{X_2})$. It follows that
$I([1]_{X_1}\times [1]_{X_2})\ge I([1]_{X_1})I([1]_{X_2})>0$. By
Lemma \ref{independence density:lemma} we can find $F\in
\Ind([1]_{X_1}\times [1]_{X_2})$ with density $I([1]_{X_1}\times
[1]_{X_2})$. Take $x\in \bigcap_{n\in F}(\sigma\times
\sigma)^{-n}([1]_{X_1}\times [1]_{X_2})$. Say, $x=(1_{K_1},
1_{K_2})$ for some $K_1, K_2\in \Z_+$. Then $1_{K_i}\in X_i$,
$i=1, 2$, and $K_1\cap K_2\supseteq F$. It follows that $K_1\cap
K_2\in \F_{\pud}$.

(2). Set $X_i=X_{S_i}$.
Since $S_i$ is in $\F_{\ps}$, there exists $1_{F_i}\in X_i$ with
$F_i\in \F_{\rms}$. By Lemma~\ref{unique minimal:lemma}.(2),
for each $i=1,2$, there is a minimal set $M_i\neq \{(0, 0, \dots)\}$ contained in
$X_i$.
Consider the t.d.s. $(M_1\times M_2, \sigma\times \sigma)$.
Let $M$ be a minimal set contained in
$M_1\times M_2$ and take $x\in M$.
Say,
$x=(1_{K_1}, 1_{K_2})$ for some $K_1, K_2\subseteq
\Z_+$. Then $K_1$ and $K_2$ are nonempty. For any $j, k\in \Z_+$, since
$\sigma^j\times \sigma^k$ is a factor map from
$M$ to a minimal set in $M_1\times M_2$, $\sigma^j\times
\sigma^k(x)=(\sigma^j1_{K_1}, \sigma^k 1_{K_2})$ is also a minimal point.
Replacing $x$ by $\sigma^{\min K_1}\times \sigma^{\min K_2}(x)$ if necessary,
we may assume that $\min K_1=\min K_2=0$.
Then $K_1\cap K_2=N(x, [1]_{X_1}\times [1]_{X_2})$ is syndetic and central.

(3). This is trivial.
\end{proof}

\begin{thm} The product of finitely many
$\F_{\rms}$-(resp. $\F_{\rs},
\F_{\pd}$) independent
t.d.s. is $\F_{\rms}$-(resp. $\F_{\rs},
\F_{\pd}$)
independent.
\end{thm}
\begin{proof} We shall prove the case $\F=\F_{\rms}$, and the proof
for the other cases is similar. Let $(X_i,T_i)$ be
an $\F_{\rms}$-independent t.d.s. for $i=1,2$. Let $U_1, \ldots, U_n$ and
$V_1,\ldots, V_n$ be nonempty open subsets of $X_1$ and $X_2$
respectively. Then there are syndetic sets $S_1\in \Ind(U_1,\ldots,
U_n)$ and $S_2\in \Ind(V_1,\ldots, V_n)$. By Lemma
\ref{productminimal} there are two subsets $K_1, K_2$ of $\Z_+$ such
that
$1_{K_i}\in X_{S_i}$, $i=1, 2$,
and $K_1\cap K_2$ is
syndetic. It is clear that $K_1\in\Ind(U_1,\ldots, U_n)$ and
$K_2\in\Ind(V_1,\ldots, V_n)$. Thus, $K_1\cap K_2\in \Ind(U_1\times
V_1, \ldots, U_n\times V_n)$. This implies that $(X_1\times X_2,
T_1\times T_2)$ is $\F_{\rms}$-independent. The theorem follows by
induction.
\end{proof}

Since  a family $\F$ has the Ramsey property if and only if its dual
family $\F^*$ has the finite intersection property,
we have

\begin{thm} \label{Ramsey to product:thm}
Let $\F$ be a family with the Ramsey property. Then the product
of finitely many $\F^*$-independent t.d.s. remains
$\F^*$-independent.
\end{thm}

In \cite[page 278]{Weiss00} Weiss constructed two weakly mixing
t.d.s. whose product is not transitive. (Weiss's example was only
stated to be $\Z$-weakly mixing, but is easily checked to be
$\Z_+$-weakly mixing.) In view of Theorem~\ref{weakmixing}, this
implies that the product of $\F_{\inf}$-independent
($\F_{\ip}$-independent resp.) t.d.s. may fail to be
$\F_{\inf}$-independent ($\F_{\ip}$-independent resp.).

\section{Classes of measurable $\F$-independence}

\subsection{General discussion}
In this subsection we characterize $\F_{\inf}$-(resp. $\F_{\ip}$, $\F_{\pubd}$)
independent m.d.s. in Theorems~\ref{wmmds} and \ref{kfpd}.

Recall that a m.d.s $(X, T)$ is said to be {\it ergodic} if for any $A, B\in \B$ with positive measures,
$N(A, B)$ is nonempty; it is called {\it weakly mixing} if $T\times T$ is ergodic.
Similar to the topological case (Theorem~\ref{weakmixing}) we have

\begin{thm}\label{wmmds}
For a m.d.s. $(X,\B, \mu, T)$ the following are equivalent:
\begin{enumerate}
\item $(X, \B, \mu, T)$ is weakly mixing.

\item $(X, \B, \mu, T)$ is $\F_{\inf}$-independent of order $2$.

\item $(X, \B, \mu, T)$ is $\F_{\inf}$-independent.

\item $(X, \B, \mu, T)$ is $\F_{\ip}$-independent of order $2$.

\item $(X, \B, \mu, T)$ is $\F_{\ip}$-independent.
\end{enumerate}
\end{thm}
\begin{proof} It is clear that (5)$\Rightarrow$(4)$\Rightarrow$(2) and
(5)$\Rightarrow$(3)$\Rightarrow$(2).
The implication (2)$\Rightarrow$(1) follows from the fact that
if for any $A, B\in \B$ with positive measures one has $N(A, B)\cap N(A, A)\neq \emptyset$,
then $(X, \B, \mu, T)$ is weakly mixing \cite[Theorem 4.31]{HF}. (\cite[Theorem 4.31]{HF}
was proved only for invertible m.d.s., as it depended on \cite[Theorem 4.30]{HF}
which in turn was only proved for invertible m.d.s.; \cite[Theorem 3.4]{Parry} proved
\cite[Theorem 4.30]{HF} for any m.d.s.,
thus \cite[Theorem 4.31]{HF} also holds
for any m.d.s..)

When $T$ is weakly mixing, so is $T\times \dots \times T$ \cite[Theorem 1.24]{PW}.
Thus the proof of (1)$\Rightarrow$(5) in Theorem~\ref{weakmixing} also applies here.
\end{proof}

It was proved in \cite[Theorem 8.3]{HY2} and \cite[Theorem 3.16]{KL2} that a t.d.s. is
topological K if and only if each of its finite covers by non-dense open subsets has positive entropy.
Moreover, it is
shown in \cite[Theorem 9.4]{HY2} that there exists t.d.s. which is $\F_{\pubd}$-independent of order $2$
but is not $\F_{\pubd}$-independent of order $3$. Now we show that in the
measurable setup the situation is different.

We refer the reader to \cite[Chapter 4]{Parry} for the basics of the entropy theory.
A m.d.s. $(X, \B, \mu, T)$ is said to have {\it completely positive
entropy} if for every non-trivial countable measurable partition
$\alpha$ of $X$ with $0<H(\alpha)<\infty$ one has $h_{\mu}(T,
\alpha)>0$. The Rohlin-Sinai theorem says that
an invertible m.d.s. has completely positive entropy if and only if
it is an K-automorphism \cite{RS} \cite[Theorem 4.12]{Parry}.

For $a \ge 2$ let $\Omega_a=\{0,1,\dots,a-1\}^\Z$ and $Y\subseteq
\Omega_a$. A subset $I\subseteq \Z$ is called an {\it
interpolating set for} $Y$ if $Y|_I=\Omega_a|_I$.
Now suppose
that $(X,\mathcal B,\mu,T)$ is an invertible m.d.s. and that $\mathcal
P=\{P_0,P_1,\dots,P_{a-1}\}$ is a finite measurable partition of
$X$. Construct a set $Y_{\mathcal P}\subseteq \Omega_a$ as follows:
$$
Y_{\mathcal P}=\{\omega\in \Omega_a:\text{for all nonempty finite subsets}\
J\subseteq \Z,\ \mu(\bigcap_{j\in J}T^{-j}P_{\omega_j})>0\}.
$$
Glasner and Weiss showed that an invertible m.d.s. $(X, \B, \mu, T)$
has completely positive entropy if and only if
for every
finite measurable partition $\P=\{P_0, P_1, \dots, P_{a-1}\}$ of $X$ with
$\min_{0\le j\le a-1}\mu(P_j)>0$ the set $Y_{\P}$ has interpolating sets of positive density.
In our terminology, clearly  interpolating
sets of $\P$ are exactly the independence sets of the tuple $(P_0, P_1, \dots, P_{a-1})$.
Now we extend the result of Glasner and Weiss to general m.d.s..


\begin{thm}\label{kfpd}
Let $(X, \B, \mu, T)$ be a m.d.s.. Then the following are equivalent:
\begin{enumerate}
\item $(X,\mathcal{B},\mu,T)$ is $\F_{\pubd}$-independent.

\item $(X,\mathcal{B},\mu,T)$ is $\F_{\pubd}$-independent of order $2$.

\item $(X,\mathcal{B},\mu,T)$ has completely positive entropy.

\end{enumerate}
\end{thm}

To prove Theorem~\ref{kfpd}, we need some preparation. For a Lebesgue space $(X, \B, \mu)$ and
a measurable partition $\alpha$ of $X$, we denote by $\hat{\alpha}$ the $\sigma$-algebra generated by
the items of $\alpha$; for a family $\{\B_j\}_{j\in J}$
of sub-$\sigma$-algebras of $\B$, we denote by
$\bigvee_{j\in J}\B_j$ the sub-$\sigma$-algebra of $\B$ generated by $\bigcup_{j\in J}\B_j$.
For a m.d.s. $(X, \B, \mu, T)$, a measurable partition $\alpha$ of $X$, and $0\le n\le m\le \infty$, we
denote $\bigvee^{m}_{j=n}T^{-j}\alpha$ and $\bigvee^{m}_{j=n}T^{-j}\hat{\alpha}$ by $\alpha^m_n$ and
$\hat{\alpha}^m_n$ respectively.
The following lemma is \cite[Lemma 4.6]{Parry} for non-invertible m.d.s..

\begin{lem} \label{Parry1:lemma}
Let $(X, \B, \mu, T)$ be a m.d.s., and let $\alpha$ and $\beta$ be countable measurable partitions
of $X$ with $H(\alpha), H(\beta)<\infty$. If $\beta\le \alpha$ or $\alpha\le \beta$, then
$\lim_{n\to +\infty}\frac{1}{n}H(\alpha^{n-1}_0|\hat{\beta}^{\infty}_n)=H(\alpha|\hat{\alpha}^{\infty}_1)$.
\end{lem}
\begin{proof} We follow the proof of \cite[Lemma 4.6]{Parry}. Consider first the case $\beta\le \alpha$. The sequence
of $\sigma$-algebras $\{\hat{\alpha}^n_1\vee \hat{\beta}^{\infty}_{n+1}\}_{n\in \N}$
is increasing
and their union generates the $\sigma$-algebra $\hat{\alpha}^{\infty}_1$.
By the increasing Martingale theorem \cite[Theorem 2.6]{Parry} one has
$$\lim_{n\to +\infty}H(\alpha|\hat{\alpha}^n_1\vee \hat{\beta}^{\infty}_{n+1})=H(\alpha|\hat{\alpha}^{\infty}_1).$$
Since
\begin{eqnarray*}
H(\alpha^{n-1}_0|\hat{\beta}^{\infty}_n)
=\sum^{n-1}_{i=0}H(T^{-i}\alpha|\hat{\alpha}^{n-1}_{i+1}\vee \hat{\beta}^{\infty}_n)
=\sum^{n-1}_{i=0}H(\alpha|\hat{\alpha}^{n-1-i}_1\vee \hat{\beta}^{\infty}_{n-i})
=\sum^{n-1}_{i=0}H(\alpha|\hat{\alpha}^i_1\vee \hat{\beta}^{\infty}_{i+1}),
\end{eqnarray*}
we conclude that $\lim_{n\to +\infty}\frac{1}{n}H(\alpha^{n-1}_0|\hat{\beta}^{\infty}_n)
=H(\alpha|\hat{\alpha}^{\infty}_1)$.

Next we consider the case $\alpha\le \beta$. One has
$$ \limsup_{n\to +\infty}\frac{1}{n}H(\alpha^{n-1}_0|\hat{\beta}^{\infty}_n)\le
\lim_{n\to +\infty}\frac{1}{n}H(\alpha^{n-1}_0|\hat{\alpha}^{\infty}_n)=H(\alpha|\hat{\alpha}^{\infty}_1),$$
where the second equality comes from the above paragraph.
One also has
\begin{eqnarray*}
 \frac{1}{n}H(\beta^{n-1}_0|\hat{\beta}^{\infty}_n)=\frac{1}{n}H(\beta^{n-1}_0\vee
 \alpha^{n-1}_0|\hat{\beta}^{\infty}_n)
 =\frac{1}{n}H(\alpha^{n-1}_0|\hat{\beta}^{\infty}_n)+\frac{1}{n}H(\beta^{n-1}_0|\hat{\alpha}^{n-1}_0\vee
 \hat{\beta}^{\infty}_n),
\end{eqnarray*}
and
\begin{eqnarray*}
 \frac{1}{n}H(\beta^{n-1}_0|\hat{\alpha}^{\infty}_n)=\frac{1}{n}H(\beta^{n-1}_0
 \vee \alpha^{n-1}_0|\hat{\alpha}^{\infty}_n)
 =\frac{1}{n}H(\alpha^{n-1}_0|\hat{\alpha}^{\infty}_n)+\frac{1}{n}H(\beta^{n-1}_0|\hat{\alpha}^{\infty}_0).
\end{eqnarray*}
Since $H(\beta^{n-1}_0|\hat{\alpha}^{n-1}_0\vee  \hat{\beta}^{\infty}_n)\le H(\beta^{n-1}_0|
\hat{\alpha}^{\infty}_0)$, we get
 \begin{eqnarray*}
 \frac{1}{n}H(\beta^{n-1}_0|\hat{\beta}^{\infty}_n)- \frac{1}{n}H(\alpha^{n-1}_0|\hat{\beta}^{\infty}_n)
 \le \frac{1}{n}H(\beta^{n-1}_0|\hat{\alpha}^{\infty}_n)-\frac{1}{n}H(\alpha^{n-1}_0|\hat{\alpha}^{\infty}_n).
\end{eqnarray*}
Taking $\limsup$ on both sides, by the above paragraph we get
$$ H(\beta|\hat{\beta}^{\infty}_1)-\liminf_{n\to +\infty}\frac{1}{n}H(\alpha^{n-1}_0|
\hat{\beta}^{\infty}_n)\le H(\beta|\hat{\beta}^{\infty}_1)-
H(\alpha|\hat{\alpha}^{\infty}_1).
$$
That is, $\liminf_{n\to +\infty}\frac{1}{n}H(\alpha^{n-1}_0|\hat{\beta}^{\infty}_n)\ge
H(\alpha|\hat{\alpha}^{\infty}_1)$. Therefore $\lim_{n\to +\infty}\frac{1}{n}H(\alpha^{n-1}_0|
\hat{\beta}^{\infty}_n)= H(\alpha|\hat{\alpha}^{\infty}_1)$
as desired.
\end{proof}

For a m.d.s. $(X, \B, \mu, T)$, denote by $\P(T)$ the {\it Pinsker
$\sigma$-algebra} of $T$ \cite[page 113]{PW}, consisting of $A\in
\B$ such that $h_{\mu}(T, \{A, X\setminus A\})=0$. For a Lebesgue
space $(X, \B, \mu)$ and sub-$\sigma$-algebras $\B_1$ and $\B_2$
of $\B$, we write $\B_1\subseteq_{\mu} \B_2$ if for every $A_1\in
\B_1$ we can find $A_2\in \B_2$ with $\mu(A_1\Delta A_2)=0$; we
write $\B_1=_{\mu}\B_2$ if $\B_1\subseteq_{\mu}\B_2$ and
$\B_2\subseteq_{\mu}\B_1$. The next theorem appeared in
\cite[12.3]{Rohlin}. For the convenience of the reader, we give a
proof here.

\begin{thm} \label{Pinsker:thm}
Let $(X, \B, \mu, T)$ be a m.d.s.. Then $\P(T)=_{\mu}
\bigvee_{\alpha} \bigcap_{n\in \Z_+}\hat{\alpha}^{\infty}_{n}$ for
$\alpha$ running over countable measurable partitions of $X$ with
$H(\alpha)<\infty$.
\end{thm}
\begin{proof} Let $\alpha, \beta$ be countable measurable partitions of $X$ with
$H(\alpha), H(\beta)<\infty$ and $\hat{\beta}\subseteq \bigcap_{n\in \Z_+}\hat{\alpha}^{\infty}_{n}$.
 For any $m\le n$ in $\Z_+$, one has $T^{-m}\hat{\beta}\subseteq \hat{\alpha}^{\infty}_n$.
Thus $\frac{1}{n}H(\beta^{n-1}_0|\hat{\alpha}^{\infty}_n\vee \hat{\beta}^{\infty}_n)=0$ for every $n\in \N$.
Taking a limit, by Lemma~\ref{Parry1:lemma} we get $H(\beta|\hat{\beta}^{\infty}_1)=0$. That is,
$h_{\mu}(T, \beta)=0$. Thus $\bigcap_{n\in \Z_+}\hat{\alpha}^{\infty}_{n}\subseteq \P(T)$.

Conversely, let $A\in \P(T)$. Set $\alpha=\{A, X\setminus A\}$.
Then $0=h_{\mu}(T, \alpha)=H(\alpha|\alpha^{\infty}_1)$. Thus
$\hat{\alpha}\subseteq_{\mu}  \hat{\alpha}^{\infty}_1$
\cite[Proposition 14.18.1]{EG}. It follows that
$T^{-n}\hat{\alpha}\subseteq_{\mu} \hat{\alpha}^{\infty}_{n+1}$
and hence
$\hat{\alpha}^{\infty}_n\subseteq_{\mu}\hat{\alpha}^{\infty}_{n+1}$
for every $n\in \Z_+$. Then for each $n\in \N$ we can find $A_n\in
\hat{\alpha}^{\infty}_{n}$ with $\mu(A\Delta A_n)=0$. Note that
$\bigcap_{n\in \N}\bigcup_{m\ge n}A_m\in \bigcap_{n\in
\Z_+}\hat{\alpha}^{\infty}_n$ and $\mu(A\Delta (\bigcap_{n\in
\N}\bigcup_{m\ge n}A_m))=0$. Therefore $\P(T)\subseteq_{\mu}
\bigvee_{\alpha}\bigcap_{n\in \Z_+}\hat{\alpha}^{\infty}_{n}$ for
$\alpha$ running over countable measurable partitions of $X$ with
$H(\alpha)<\infty$.
\end{proof}

The next result appeared implicitly in \cite[13.2]{Rohlin}. For completeness, we give a proof here.

\begin{thm}\label{CPE:thm}
Let $(X, \B, \mu, T)$ be a m.d.s.. Then the following are equivalent:
\begin{enumerate}
%
%
\item $(X,\mathcal{B},\mu,T)$ has completely positive entropy.

\item For every countable measurable partition $\alpha$ of $X$ with $H(\alpha)<\infty$,
one has $\lim_{n\to +\infty}h_{\mu}(T^n, \alpha)=H(\alpha)$.
\end{enumerate}
\end{thm}
\begin{proof}
(1)$\Rightarrow$(2): Let $\alpha$ be a countable measurable
partition of $X$ with $H(\alpha)<\infty$. For each $n\in \N$ one
has $h_{\mu}(T^n, \alpha)=H(\alpha|\bigvee^{\infty}_{j=1}
T^{-jn}\hat{\alpha})\ge H(\alpha|\hat{\alpha}^{\infty}_n)$. By the
decreasing Martingale theorem \cite[Theorem 14.28]{EG} we have
$\lim_{n\to +\infty}
H(\alpha|\hat{\alpha}^{\infty}_n)=H(\alpha|\bigcap_{n\in
\Z_+}\hat{\alpha}^{\infty}_n)$. Since $T$ has completely positive
entropy, $\P(T)$ is exactly the $\sigma$-algebra of measurable
subsets of  $X$ with measure $0$ or $1$. Thus
$H(\alpha|\bigcap_{n\in \Z_+}\hat{\alpha}^{\infty}_n)=H(\alpha)$
by Theorem~\ref{Pinsker:thm}. Therefore $\liminf_{n\to
+\infty}h_{\mu}(T^n, \alpha)\ge H(\alpha)$. On the other hand, for
each $n\in \N$ one has $h_{\mu}(T^n, \alpha)\le H(\alpha)$. Thus
$\lim_{n\to +\infty}h_{\mu}(T^n, \alpha)=H(\alpha)$.

(2)$\Rightarrow$(1): Let $\alpha$ be a countable measurable partition of $X$ with $0<H(\alpha)<\infty$.
For each $n\in \N$ one has $h_{\mu}(T, \alpha)\ge \frac{1}{n}h_{\mu}(T^n, \alpha)$.
Since $\lim_{n\to +\infty}h_{\mu}(T^n, \alpha)=H(\alpha)>0$,
we conclude that $h_{\mu}(T, \alpha)>0$.
\end{proof}

\begin{lem} \label{indep to positive entropy:lem}
A non-trivial m.d.s. which is $\F_{\pubd}$-independent of order $2$ has positive entropy.
\end{lem}
\begin{proof}
Assume that $(X,\mathcal{B},\mu,T)$ is a non-trivial m.p.s. which is
$\F_{\pubd}$-independent of order $2$ and has entropy $0$.
Clearly $(X,\mathcal{B},\mu,T)$ is ergodic.
By Rosenthal's extension of the  Jewett-Krieger theorem
to non-invertible m.d.s. \cite{Ros88}, there exists a
t.d.s. $(\widehat{X},\widehat{T})$ with a unique invariant Borel
probability measure $\widehat{\mu}$ such that $\widehat{\mu}$ has full support and
the m.d.s. $(X, \B, \mu, T)$ and $(\widehat{X}, \B_{\widehat{X}}, \widehat{\mu}, \widehat{T})$
are isomorphic, where $\B_{\widehat{X}}$ denotes the Borel $\sigma$-algebra of $\widehat{X}$,
in the sense that there are $X_0\in \B$, $\widehat{X}_0\in \B_{\widehat{X}}$ and a measure-preserving
bijection $\phi:X_0\rightarrow \widehat{X}_0$ with
$\mu(X_0)=\widehat{\mu}(\widehat{X}_0)=1$, $TX_0\subseteq X_0$, $\widehat{T}\widehat{X}_0\subseteq \widehat{X}_0$,
and $\phi\circ T=\widehat{T}\circ \phi$.
By the variational principle \cite[Theorem 8.6]{PW}, one has
$$h_{\rm top}(\widehat{T})=h_{\widehat{\mu}}(\widehat{T})=h_{\mu}(T)=0.$$

Since
$(X,\mathcal{B},\mu, T)$ is non-trivial,
$(\widehat{X},\widehat{T})$ is a non-trivial t.d.s..
Thus we can find two disjoint closed subsets $A,B$ of $\widehat{X}$ with
$\widehat{\mu}(A)>0$, $\widehat{\mu}(B)>0$. Set
$\mathcal{U}=\{\widehat{X}\setminus A,\widehat{X}\setminus B\}$.
Then $\mathcal{U}$ is an open cover of $\widehat{X}$, and for any
$F\in \Ind(A,B)$ we have
\begin{align*}
0&=h_{\rm top}(\widehat{T})\ge
h_{\rm top}(\widehat{T},\mathcal{U})=\lim_{n\to
+\infty}\frac{1}{n}
\log N(\bigvee_{i=0}^{n-1}\widehat{T}^{-i}\mathcal{U})\\
&\ge \limsup_{n\to +\infty}\frac{1}{n}
\log N(\bigvee_{i\in F\cap \{ 0,1,\dots,n-1\}}\widehat{T}^{-i}\mathcal{U})\\
&=\limsup_{n\to +\infty}\frac{1}{n} \log 2^{|F\cap \{
0,1,\dots,n-1\}|}=\overline{d}(F)\log 2.
\end{align*}
Hence $\overline{d}(F)=0$. Thus $\Ind(A, B)\cap \F_{\pud}=\emptyset$. It
follows from Example~\ref{bF:ex} and Proposition~\ref{same indep:prop}
that $\Ind(A,B)\cap \mathcal{F}_{\pubd}=\emptyset$.
Thus $(\widehat{X}, \B_{\widehat{X}}, \widehat{\mu}, \widehat{T})$ is not $\F_{\pubd}$-independent of
order $2$. Then by Remark~\ref{measure algebra:rem} $(X, \B, \mu, T)$ is not
$\F_{\pubd}$-independent of order $2$ either.
\end{proof}

We shall need the following consequence of Karpovsky and Milman's
generalization of the Sauer-Perles-Shelah lemma \cite{KM,Sau,SS}.

\begin{lem} \label{KM-lemma} (\cite{KM}). Given $r\ge 2$ in $\N$ and $\lambda> \ln(r-1)$
there exists a  constant $c>0$ depending only on $r$ and $\lambda$ such
that, for all $n\in \N$ and $S\subseteq \{1, 2,\dots,
r\}^{\{0,1,\dots,n-1\}}$ satisfying $|S|\ge e^{\lambda n}$ there is
an $I\subseteq \{0,1, 2,\dots, n-1\}$ with $|I|\ge cn$ and $S|_I =
\{1, 2, \dots , r\}^I$.
\end{lem}

Now we are ready to prove Theorem~\ref{kfpd}.

\begin{proof}[Proof of Theorem~\ref{kfpd}] When $(X,\mathcal{B},\mu,T)$ is a trivial system, this is obvious.
So we suppose that $(X,\mathcal{B},\mu,T)$ is non-trivial.
(1)$\Rightarrow$(2) is obvious.



(3)$\Rightarrow$(1): We claim first that $(X, \B, \mu)$ is {\it
non-atomic} in the sense that $\mu(\{x\})=0$ for every $x\in X$.
In fact, since $T$ has completely positive entropy, it is ergodic.
If $\mu(\{x\})>0$ for some $x\in X$, then we can find some $n\in
\N$ such that $x, Tx, \dots, T^{n-1}x$ are pairwise distinct,
$T^nx=x$, and $\mu(x)=\mu(Tx)=\dots=\mu(T^{n-1}x)=\frac{1}{n}$.
If $n>1$, denoting
by $\beta$ the partition of $X$ into $\{x\}$ and its complement, we have $h_{\mu}(T, \beta)=0$.
Thus $n=1$, which means that $(X, \B, \mu, T)$ is trivial. Therefore $(X, \B, \mu)$ is non-atomic.

Given a tuple
$(A_1,\dots,A_k)$ of sets in $\mathcal{B}$ with positive measures,
we are going to show that $\Ind(A_1,\dots,A_k)\cap
\mathcal{F}_{\pubd}\neq \emptyset$. Without loss of generality, we
may assume that $A_1, \dots, A_k$ are pairwise disjoint.

Every non-atomic Lebesgue space is isomorphic to the closed unit interval endowed with its Borel $\sigma$-algebra and
the Lebesgue measure \cite[Theorem 17.41]{K1}.
It follows that there exist an $r\in \mathbb{N}$ and a
measurable partition $\alpha=\{B_1,\dots,B_r\}$ of $X$ such that $r>
k$, $\mu(B_i)=\frac{1}{r}$ for $i=1,2,\dots,r$, and $B_j$ is a
subset of $A_j$ for $j=1,2,\dots,k$. To show $\Ind(A_1,\dots,A_k)\cap
\mathcal{F}_{\pubd}\neq \emptyset$, it is sufficient to show
$\Ind(B_1,\dots,B_r)\cap \mathcal{F}_{\pubd}\neq \emptyset$.

By Theorem~\ref{CPE:thm} we have  $\lim_{n\to +\infty} h_\mu(T^n,\alpha)=H(\alpha)=\ln r$. Thus
there exists $\ell \in \mathbb{N}$ such that
$\lambda:=h_\mu(T^\ell,\alpha)>\ln(r-1)$. Then
$\frac{1}{n}H(\bigvee_{i=0}^{n-1}T^{-\ell i}\alpha)\ge
\lambda>\ln(r-1)$ for all $n\in \mathbb{N}$. For any given finite
measurable partition $\beta$ of $X$, we define
$$|\beta|_\mu=| \{B\in \beta: \mu(\beta)>0\}|.$$
Then $|\bigvee_{i=0}^{n-1}T^{-\ell i}\alpha|_\mu\ge
e^{H(\bigvee_{i=0}^{n-1}T^{-\ell i}\alpha)}\ge e^{\lambda n}$
for all $n\in \mathbb{N}$.

Now combing this with Lemma \ref{KM-lemma}, we see that there exists a constant
$c>0$ depending on only $r$ and $\lambda$ such that, for any $n\in
\mathbb{N}$ there is an $I_n\subseteq \{0,1, 2,\dots, n-1\}$ with
$|I_n|\ge cn$ and $\mu(\bigcap_{i\in I_n}T^{-\ell i}B_{s(i)})>0$ for
any $s\in\{1, 2, \dots , r\}^{I_n}$.
This implies that $\ell I_n\in
\Ind(B_1, \dots,B_r)$ for each $n\in \N$.
From Lemma~\ref{independence density:lemma} we conclude that
$\Ind(B_1, \dots, B_r)\cap \F_{\pd}\neq \emptyset$.

(2)$\Rightarrow$(3): Assume that $(X,\mathcal{B},\mu,T)$ is
$\F_{\pubd}$-independent of order $2$.
Note that the definitions of independence sets and entropy apply to more general
measure-theoretical dynamical systems in which the probability space does not have to be a Lebesgue
space.
In this sense $(X, \P(T), \mu, T)$ is also $\F_{\pubd}$-independent of order $2$
and has entropy $0$.
Since $(X, \B, \mu)$ is a Lebesgue space, it is easy to see
that $\B$ is separable under the semi-metric $d(A, B)=\mu(A\Delta B)$.
Then $\P(T)$ is also separable under this semi-metric.
It follows that there is a m.d.s. $(Y, \J, \nu, S)$ (i.e., $(Y, \J, \nu)$ is a Lebesgue space) such that
the measure algebra triples associated to
$(X, \P(T), \mu, T)$ and $(Y, \J, \nu, S)$ in Remark~\ref{measure algebra:rem} are isomorphic
\cite[Proposition 5.3]{HF}.
Then $(Y, \J, \nu, S)$ is also $\F_{\pubd}$-independent of order $2$
and has entropy $0$.
By Lemma~\ref{indep to positive entropy:lem} $(Y, \J, \nu, S)$ is trivial. Thus
$\P(T)$ consists of measurable subsets of $X$ with measure $0$ or $1$. That is, $(X, \B, \mu, T)$
has completely positive entropy.
\end{proof}

\subsection{Non-existence of $\F_{\rms}$-independent m.d.s.}
In this subsection we establish the somewhat surprising result that there is no non-trivial m.d.s. which is
$\F_{\rms}$-independent. In the following, we aim to show
that for any non-periodic m.d.s. $(X,\mathcal{B},\mu,T)$,
there exists $A\in \mathcal{B}$ with $\mu(A)>0$ such that $\Ind(A)$
does not contain a syndetic set.

A m.d.s. $(X, \B, \mu, T)$
is called {\it non-periodic} or {\it free} if $\mu(\{x\in X:
T^nx=x\})=0$ for every $n\in \N$.
It is easy to see that
an ergodic m.d.s. $(X, \B, \mu, T)$
is non-periodic if and only if $(X, \B,
\mu)$ is {\it non-atomic} in the sense that $\mu(x)=0$ for every $x\in X$.

\begin{thm} \label{notsyndetic}
Let $(X, \B, \mu, T)$ be a non-periodic m.d.s..
Then for any $\varepsilon>0$
there exists $A\in \B$ with $\mu(A)>1-\varepsilon$ such that
$\Ind(A)$ does not contain any syndetic set.
\end{thm}
\begin{proof}
Endow $X$ with a Polish topology such that $\B$ is the corresponding Borel $\sigma$-algebra.
Replacing the Polish topology on $X$ by a
finer one if necessary \cite[Theorem 13.11 and Lemma 13.3]{K1}, we may
assume that $T$ is continuous.
Let $\varepsilon>0$. We claim that there is a compact subset $K$ of $X$ such that $\mu(K)>1-\varepsilon$ and
$A_n:=\bigcap_{j\in
\Z_+}\bigcup^{n-1}_{i=0}T^{-j-i}K$ has measure $0$ for every $n\in \N$. Assuming this claim
let us show how it implies the theorem.

Since $\mu(\bigcup_{n\in \N}A_n)=0$, one has $\mu(K\setminus
(\bigcup_{n\in \N}A_n))=\mu(K)>1-\varepsilon$. By the regularity of
$\mu$ \cite[Theorem 17.11]{K1},  we can find a compact set $A$ contained in $ K\setminus
(\bigcup_{n\in \N}A_n)$ such that $\mu(A)>1-\varepsilon$. We shall show that $\Ind(A)$ does not
contain any syndetic set.

Let $F\in \Ind(A)$ be nonempty. Replacing $F$ by $F-\min F$ if
necessary, we may assume that $0\in F$. One has $\mu(\bigcap_{j\in J}T^{-j}A)>0$ and
hence $\bigcap_{j\in
J}T^{-j}A\neq \emptyset$ for every nonempty finite subset $J$ of
$F$. Since $A$ is compact, we conclude that $\bigcap_{j\in
F}T^{-j}A$ is nonempty. Take $x\in \bigcap_{j\in F}T^{-j}A$. Then
$x\in A$ and $T^jx\in A\subseteq  K$ for every $j\in F$. For each $n\in \N$ one
has $x\not\in A_n$, and hence for some $j_n\in \Z_+$ none of
$T^{j_n}x, T^{j_n+1}x, \dots, T^{j_n +n-1}x$ is in $K$. Then
$[j_n, j_n+n-1]\cap F=\emptyset$. Therefore $F$ is not
syndetic.

We are left to prove the above claim. Since the main idea of the proof is well
illustrated in the case $\mu$ is ergodic, we consider this case first.

So assume that $\mu$ is ergodic. Since $(X, \B, \mu, T)$
is non-periodic, by the comment before Theorem~\ref{notsyndetic},  $(X, \B,
\mu)$ is non-atomic. Replacing $X$ by $\supp(\mu)$ if necessary, we may assume that $\mu$ has full support.
Take $x\in X$ and set $W=\{T^nx: n\in \Z_+\}$. Then $TW\subseteq W$, and $W$ is nonempty and countable.
Since $\mu$ is non-atomic, one has $\mu(W)=0$, and hence $\mu(X\setminus W)=1$.
By the regularity of $\mu$,  we can
find a compact set $K$ contained in $X\setminus W$ such that
$\mu(K)>1-\varepsilon$.
For any $n\in
\N$, $\bigcup^n_{i=0}T^{-i}K$ is a closed subset of $X$ with $\bigcup^n_{i=0}T^{-i}K\neq X$,
since $W\cap (\bigcup^n_{i=0}T^{-i}K)=\emptyset$. As $\mu$ has full support,
$\mu(\bigcup^n_{i=0}T^{-i}K)<1$ for all $n\in \N$. Note that $A_n\in \B$ and $T^{-1}A_n\supseteq A_n$.
Since $\mu$ is ergodic and $\mu(A_n)\le \mu(\bigcup^n_{i=0}T^{-i}K)<1$, we get $\mu(A_n)=0$.
This finishes the proof in the case
$\mu$ is ergodic.

Now we consider the general case, using the ergodic decomposition of $(X, \B, \mu, T)$.

Denote by $P(X)$ the set of all probability Borel measures on $X$, and endow it with the $\sigma$-algebra
generated by the functions $\mu'\mapsto \mu'(A)$ on $P(X)$ for all $A\in \B$ \cite[Section 17.E]{K1}.

From the ergodic decomposition of $(X, \B, \mu, T)$ we know that
there exist a set $X'\in \B$ with $\mu(X')=1$ and $TX'\subseteq
X'$, a Lebesgue space $(Y, \J, \nu)$, a measurable map $\pi:
X'\rightarrow Y$, a measurable map $y\mapsto \mu_y$ from $Y$ to
$P(X)$, and a set $Y'\in \J$ with $\nu(Y')=1$ such that $\pi
T=\pi$, $\pi\mu=\nu$, $\mu(A)=\int_Y \mu_y(A)\, d\nu(y)$ for all
$A\in \B$, and $\mu_y(\pi^{-1}(y))=1$ and $T\mu_y=\mu_y$ and $(X,
\B, \mu_y, T)$ is ergodic for every $y\in Y'$ \cite[Theorem
3.42]{EG}. (\cite[Theorem 3.42]{EG} was only proved for invertible
m.d.s., but it is easy to see that the proof works for any
m.d.s..)

Set $W_1=\{x\in X: T^nx=x \text{ for some } n\in \N\}$. Clearly $W_1$ is in $\B$.
By assumption $0=\mu(W_1)=\int_Y\mu_y(W_1)\, d\nu(y)$. Thus $\mu_y(W_1)=0$ for $\nu$ a.e. $y\in Y$.
Replacing $Y'$ by a smaller measurable set if necessary, we may assume that $\mu_y(W_1)=0$ for every $y\in Y'$.
Since $(X, \B, \mu_y, T)$ is ergodic for every $y\in Y'$, it follows that $\mu_y$ is non-atomic for every
$y\in Y'$.


Endow $Y$ with a Polish topology such that $\J$ is the corresponding Borel $\sigma$-algebra.
Replacing the Polish topology on $X$ by a
finer one if necessary, we may
assume that $\pi$ is continuous.

Denote by $F(X)$ the set of all closed subsets of $X$, and endow it with the Effros Borel structure, i.e.,
the $\sigma$-algebra generated by the sets $\{Z\in F(X): Z\cap U\neq \emptyset\}$ for all
open subsets $U$ of $X$.
The map $\phi: P(X)\rightarrow F(X)$ sending
each $\mu'$ to $\supp(\mu')$ is measurable \cite[Exercise 17.38]{K1}.
By the Kuratowski-Ryll-Nardzewski selection theorem \cite[Theorem 12.13]{K1} we can find
a measurable map $\psi: F(X)\rightarrow X$ such that $\psi(Z)\in Z$ for each nonempty $Z\in F(X)$.

Note that $\supp(\mu_y)\subseteq \pi^{-1}(y)$ for every $y\in Y'$.
Thus the map $\varphi_n: Y'\rightarrow X$ sending $y$ to
$T^n(\psi(\phi(\mu_y)))$ is measurable and injective for each
$n\in \Z_+$. Recall that a measurable space is a {\it standard
Borel space} if the $\sigma$-algebra is the Borel $\sigma$-algebra
for some Polish topology on the set. A measurable subset of a
standard Borel space together with the restriction of the
$\sigma$-algebra to the subset is also a standard Borel space
\cite[Corollary 13.4]{K1}. Thus $Y'$ together with the restriction
of $\J$ on $Y'$ is a standard Borel space. The Lusin-Souslin
theorem says that the image of any injective measurable map from a
standard Borel space to another standard Borel space is measurable
\cite[Corollary 15.2]{K1}. Thus the set $W:=\bigcup_{n\in
\Z_+}\varphi_n(Y')$ is in $\B$. Note that $TW\subseteq W$,  and
$W\cap \supp(\mu_y)$ is nonempty and countable for every $y\in
Y'$.

Since $\mu_y$ is non-atomic for every $y\in Y'$, one has
$\mu_y(W)=0$ for every $y\in Y'$. Thus $\mu(W)=\int_Y\mu_y(W)\,
d\nu(y)=0$, and hence $\mu(X\setminus W)=1$. By the regularity of
$\mu$, we can find a compact set $K$ contained in $X\setminus W$
such that $\mu(K)>1-\varepsilon$. For any $n\in \N$ and $y\in Y'$,
$\supp(\mu_y)\cap (\bigcup^n_{i=0}T^{-i}K)$ is a closed subset of
$\supp(\mu_y)$ with $\supp(\mu_y)\cap (\bigcup^n_{i=0}T^{-i}K)\neq
\supp(\mu_y)$, since $W\cap (\bigcup^n_{i=0}T^{-i}K)=\emptyset$
and $W\cap \supp(\mu_y)\neq \emptyset$. Thus
$\mu_y(\bigcup^n_{i=0}T^{-i}K)<1$ for all $n\in \N$ and $y\in Y'$.

We still have $A_n\in \B$ and $T^{-1}A_n\supseteq A_n$.
For each $y\in Y'$, since $(X, \B, \mu_y, T)$ is ergodic, $\mu_y(A_n)$ is
equal to either $0$ or $1$.
By the above paragraph we have
$\mu_y(A_n)\le \mu_y(\bigcup^{n-1}_{i=0}T^{-i}K)<1$ for each $y\in Y'$. Thus
$\mu_y(A_n)=0$ for each $y\in Y'$. Therefore $\mu(A_n)=\int_Y\mu_y(A_n)\, d\nu(y)=0$,
as desired. This proves the claim and finishes the proof of the theorem.
\end{proof}

Now we are able to show
\begin{thm}\label{nofs} There is no non-trivial m.d.s. which is
$\F_{\rms}$-independent of order $2$.
\end{thm}
\begin{proof} Assume the contrary that there exists such a system $(X, \B, \mu, T)$.
By Theorem~\ref{wmmds}, $T$ is weakly mixing.

By Theorem~\ref{notsyndetic}, $T$ is a.e. periodic. Then
the set $A_n=\{x\in X: T^nx=x, T^jx\neq x \text{ for all } 1\le j<n\}$ has positive measure for some $n\in \N$.
Note that $TA_n=A_n$. By \cite[page 70]{Halmos} we can find
$B\subseteq A_n$ such that $B\in \B$, $\mu(B)=\mu(A_n)/n$, and $B, TB, \dots, T^{n-1}B$ are pairwise disjoint.
If $n\ge 2$, then  $N(B, B)\cap N(B, TB)= \emptyset$, contradicts that $(X, \B, \mu, T)$ is weakly mixing.
Thus $\mu(A_n)=0$ for every $n\ge 2$. Then $\mu(A_1)=1$. Since $(X, \B, \mu, T)$ is non-trivial, we can find
some $B\subseteq A_1$ such that $B\in \B$ and $0<\mu(B)<1$. Then $N(B, X\setminus B)=\emptyset$, again contradicting
that $(X, \B, \mu, T)$ is weakly mixing.
%
\end{proof}

\begin{rem} \label{syndetic mds:rem}
Using Theorem~\ref{no minimal syndetic tds:thm} one can strengthen
Theorem~\ref{nofs} as follows. For any nontrivial ergodic m.d.s.
$(X, \B, \mu, T)$, Rosenthal's extension of the  Jewett-Krieger
theorem to non-invertible m.d.s. \cite{Ros88} says that there
exists a t.d.s. $(\widehat{X},\widehat{T})$ with a unique
invariant Borel probability measure $\widehat{\mu}$ such that
$\widehat{\mu}$ has full support and the m.d.s. $(X, \B, \mu, T)$
and $(\widehat{X}, \B_{\widehat{X}}, \widehat{\mu}, \widehat{T})$
are isomorphic, where $\B_{\widehat{X}}$ denotes the Borel
$\sigma$-algebra of $\widehat{X}$. Then
$(\widehat{X},\widehat{T})$ is minimal (see for example
\cite[Theorem 6.17]{PW}). Furthermore, the proof in \cite{Ros88}
shows that we can choose $\widehat{X}$ to be a Cantor set. For any
real-valued continuous function $f$ on $X$, the sequence
$\{\frac{1}{n+1}\sum_{i=0}^nf\circ \widehat{T}^i\}_{n\in \Z_+}$ of
functions on $X$ converges to the constant function
$\int_{\widehat{X}}f(x)\, d\widehat{\mu}(x)$ uniformly as $n\to
+\infty$ \cite[Theorem 6.19]{PW}. By Theorem~\ref{no minimal
syndetic tds:thm} we can find disjoint nonempty clopen subsets
$\widehat{V}_0$ and $\widehat{V}_1$ of $\widehat{X}$ such that
$\Ind(\widehat{V}_0, \widehat{V}_1)\cap \F_{\rms}=\emptyset$. Say,
$\widehat{V}_j$ corresponds to $V_j\in \B$ for $j=1, 2$. Then
$V_0$ and $V_1$ are disjoint and have positive measures, and
$\Ind(V_0, V_1)\cap \F_{\rms}=\emptyset$. Furthermore, taking $f$
to be $1_{\widehat{V}_j}$,  we see that the sequence
$\{\frac{1}{n+1}\sum_{i=0}^n1_{V_j}\circ T^i\}_{n\in \Z_+}$
converges to $\mu(V_j)$ in $L^{\infty}(X, \mu)$ for $j=1, 2$.
\end{rem}

\subsection{Finite product}


By contrast to the topological case, it is well known that the
product of two weakly mixing m.d.s. is still weakly mixing \cite[Proposition 4.6]{HF}. In view
of Theorem~\ref{wmmds}, this means that the products of
finitely many $\F_{\inf}$-independent ($\F_{\ip}$-independent resp.) m.d.s. are
$\F_{\inf}$-independent ($\F_{\ip}$-independent resp.).

Meanwhile, it is known that the product of finitely many
invertible completely positive entropy m.d.s. has completely
positive entropy \cite[Theorem 4.14]{Parry}. As the topological
case, every m.d.s. has a {\it natural extension} \cite[Page
240]{CFS}, which is always invertible. The natural extension of a
completely positive entropy m.d.s. has completely positive entropy
\cite[13.8]{Rohlin} (one can also deduce this from
Theorem~\ref{CPE:thm} and the fact that the natural extension of a
m.d.s. is the inverse limit of a sequence of m.d.s.
identical to the original one). It follows that
  the product of finitely many
 completely positive entropy m.d.s. has completely positive entropy.
In view of Theorem~\ref{kfpd}, this means that the product of finitely
many $\F_{\pd}$-independent m.d.s. remains $\F_{\pd}$-independent. Thus we
make the following conjecture.

\begin{conj} For any  family $\F$,  the product of finitely many
$\F$-independent m.d.s. remains $\F$-independent.
\end{conj}

\section{Topological proof of minimal topological K systems are strongly
mixing}\label{c7}

In this section we prove Theorem~\ref{topological K to strongly mixing:thm} and
Corollary~\ref{upe to transitive:cor}.

For a cover $\mathcal{V}$ of a compact space $X$
by open subsets, we denote by $N(\mathcal{V})$ the minimal cardinality
of subcovers of $\mathcal{V}$.
Let $\F$ be a family.
A t.d.s. $(X,T)$ is called {\it $\F$-scattering} if for each
$F=\{a_1<a_2<\ldots\}\in\F$ and each finite cover $\mathcal U$
of $X$ by non-dense open subsets, one has $\lim_{n\to +\infty}
N(\bigvee_{i=1}^nT^{-a_i}\mathcal{U})=\infty$.
It was shown in \cite[Theorem 5.5]{HSY} using ergodic theory that topological K systems are
$\F_{\inf}$-scattering.  Combining this with the fact that a minimal
$\F_{\inf}$-scattering t.d.s. is
strongly mixing \cite[Theorem 5.6]{HY1},  one knows
that a minimal topological K system is strongly mixing \cite[Theorem 5.10]{HSY}.
Now we give a topological proof of the fact that a topological K
system is $\F_{\inf}$-scattering.

Recall that for
any $F=\{a_1<a_2<\ldots\}\in \F_{\inf}$ and any open cover
$\U$ of $X$, the {\it topological sequence entropy} of $T$ with respect to $F$ is defined
as
\[ h^{F}_{\top}(T,\mathcal{U})=\limsup_{n\to +\infty}\frac{\log N(\bigvee_{i=1}^nT^{-a_i}\U)}{n}.\]

\begin{thm} \label{topological K to strongly mixing:thm}
Let $(X,T)$ be a t.d.s., $n\ge 2$, $(x_1,\ldots, x_n)$ be an
$\F_{\pubd}$-independent tuple of $X$ with points pairwise distinct,  and
$U_1,\ldots, U_n$ be pairwise disjoint closed neighborhoods of
$x_1,\ldots,x_n$ respectively. Set $\mathcal{U}=\{U_1^c,\ldots, U_n^c\}$.
Then for any $F\in \F_{\inf}$, one has $h^{F}_{\top}(T,\mathcal{U})>0$.
Consequently, a topological K system is $\F_{\inf}$-scattering.
\end{thm}
\begin{proof}Since $(x_1,\ldots, x_n)$ is an $\F_{\pubd}$-independent tuple,
there exists an $S\in \Ind(U_1,\ldots, U_n)$ with positive upper Banach density $d$.
Let $F=\{a_1<a_2<\dots\}$ in $\F_{\inf}$.
Then by Lemma~\ref{inf-B} for any $k\in \N$, setting $q_k$ to be the smallest integer no less
than $2k/d$, we can find
$p_k\in\Z$ and $W_k\subseteq \{a_1, a_2, \ldots, a_{q_k}\}$ with $|W_k|=k$ and $p_k+W_k\subseteq S$.
Thus, $W_k\in \Ind(U_1,\ldots,
U_n)$. This implies that $$h^{F}_{\top}(T,\mathcal{U})\ge
\limsup_{k\to +\infty}\frac{1}{q_k}\log N(\bigvee_{j\in W_k}T^{-j}\mathcal{U})\ge
\limsup_{k\to +\infty}
\frac{1}{q_k}\log (\frac{n}{n-1})^k=\frac{2}{d}\log \frac{n}{n-1}>0.$$

Now for any finite open cover $\mathcal{V}$
of $X$ by non-dense open subsets, we may
find some $n\ge 2$, pairwise distinct $x_1,\ldots, x_n$ in $X$ and pairwise disjoint closed neighborhoods
$U_1,\ldots, U_n$ of $x_1,\ldots, x_n$ respectively such that $\mathcal{V}$ refines
$\U=\{U^c_1, \dots, U^c_n\}$.
If $(X,T)$ is topological K, then each tuple in $X$ is $\F_{\pubd}$-independent.
Thus for any $F=\{a_1<a_2<\dots\}$ in $\F_{\inf}$, by the above paragraph we have
$h^{F}_{\top}(T,\mathcal{V})\ge h^{F}_{\top}(T,\mathcal{U})>0$.
This implies that $N(\bigvee_{i=1}^m
T^{-a_i}\mathcal{V})\to \infty$ as $m\to +\infty$, i.e., $(X, T)$ is
$\F_{\inf}$-scattering.
\end{proof}

Let $\F$ be a family. A t.d.s. $(X, T)$ is called {\it $\F$-transitive}
if for any nonempty open subsets $U$ and $V$ of $X$, one has
$N(U, V)\in \F$; it is called {\it mildly mixing} if its product with
any transitive t.d.s. is transitive.
It was shown in \cite[Theorem 7.5]{HY2} that a u.p.e. system is mildly mixing.
From \cite[Theorem 7.3]{HY2} or \cite[Theorem 3.16]{KL2} one knows that
a t.d.s. is u.p.e. if and only if it is $\F_{\pubd}$-independent of order $2$.
Denote by $\Delta$ the family in $\Z_+$ generated by the sets $F-F:=\{a-b: a,b \in F, a-b>  0\}$
for all $F\in \F_{\inf}$. By \cite[Theorem 6.6]{HY1} every $\Delta^*$-transitive
system is mildly mixing. Now we strengthen the above result to show
that every t.d.s. being $\F_{\pubd}$-independent of order $2$ is $\Delta^*$-transitive.
For this we need the following proposition, which appeared in \cite[page 84]{HF} (see also
\cite[Proposition 2.3]{YZ1}) and also follows directly from Lemma~\ref{inf-B}.

\begin{prop}\label{pubd}
If $F\in \F_{\pubd}$,  then $F-F$ is
in $\Delta^*$.
\end{prop}

\begin{cor} \label{upe to transitive:cor}
A  t.d.s. being $\F_{\pubd}$-independent of order $2$ is $\Delta^*$-transitive.
\end{cor}
\begin{proof} Let $(X,T)$ be $\F_{\pubd}$-independent of order $2$. Then for any nonempty open subsets
$U$ and $V$ of $X$, there exists an $F\in \Ind(U, V)\cap \F_{\pubd}$.
Clearly $F-F \subseteq N(U,V)$. By Proposition
\ref{pubd} one has $N(U,V)\in \Delta^*$.
\end{proof}

To end this section we make the following remark. Denote by $\F_{\rm ss}$ the family
consisting of
$S\subseteq\Z_+$ satisfying that for each  $F\in\F_{\inf}$ and
each $k\in \N$ there exists $p_k\in \Z$ with $|F\cap
(S+p_k)|\ge k$.
It is clear that for any $S\in \F_{\rm ss}$ one has
$S-S\in \Delta^*$.

\begin{rem}
One obvious corollary of Lemma \ref{inf-B} is that  $\F_{\pubd}\subseteq \F_{\rm ss}$. We
remark that there exists an $S\in \F_{\rm ss}$ containing no arithmetic
progression of length $3$ (and thus having zero upper Banach density by Roth's theorem \cite{Roth}).
\end{rem}
\begin{proof}
For $k\ge 3$ set
$$S_k=\{\{a_1,a_2,\ldots,a_k\}\subseteq \N: a_j-a_i>a_i-a_s>0\ \text{for all}\ 1 \le s<i<j\le k\}.$$

Each $S_k$ is countable. Enumerate $\bigcup_{k\ge 3} S_k$ as $\{A_1,A_2,\ldots\}$. Now let
$\{t_i\}_{i\in \N}$ be a sequence in $\N$ and set $S=\bigcup_{i\in \N}(A_i+t_i)$.

Now assume that
$F$ is an infinite subset of $\Z$. For each
$k\ge 3$, inductively we can find $b_1, b_2, \dots, b_k\in F$
such that $b_j-b_i>b_i-b_s>0$ for all $1 \le s<i<j\le k$.
This implies
that there exists $p_k\in \Z$ with $|F\cap (S+p_k)|\ge k$. Thus $S$ is in $\F_{\rm ss}$.

If we choose $t_i$ to  grow rapidly enough, it is easy to
check that $S$ does not contain any arithmetic progression of length $3$.
\end{proof}

We remak that the set of prime numbers is not in $\F_{\rm ss}$. In fact
any $S=\{a_1<a_2<\ldots\}\in \F_{\inf}$ with $a_{i+1}-a_i\to +\infty$
as $i\to +\infty$ is not in $\F_{\rm ss}$. Actually one can find an
 $F\in \F_{\inf}$  such that $|F\cap (S+p)|\le 2$
for all $p\in \Z$. To find such an $F$, start with any $b_1<b_2$ in $\N$.
Since $a_{i+1}-a_i\to +\infty$ as $i\to +\infty$, there are only finitely many $i$
and $j$
satisfying $a_j-a_i=b_2-b_1.$ From this we can find $b_3>b_2$ in $\N$
such that $|\{b_1, b_2, b_3\}\cap(S+p)|\le 2$ for every
$p\in \Z$. Inductively we find $b_3<b_4<b_5<\ldots$ in $\N$ such
that $|\{b_1, b_2, \dots, b_k\}\cap(S+p)|\le 2$ for every
$k\in \N$ and $p\in \Z$.

\section{Appendix}

In this appendix we prove Theorem~\ref{infinite words:lemma}.

\begin{proof}[Proof of Theorem~\ref{infinite words:lemma}]
Take $d\in \N$ with $p^d>\frac{(d+p)\ell}{p}$. Let $m>d$ be large
enough, which we shall determine later. It suffices to show that
for every $n\in \N$ there exists $x_n\in \Sigma_p$ with $x_n[j,
j+m-1]\not \in A_j$ for all $0\le j\le n$. Then any limit point of
the sequence $\{x_n\}_{n\in \Z_+}$ in $\Sigma_p$ satisfies the
requirement.

As convention, set $\Lambda_p^0$ to be the one element set consisting of the empty word.
For any $k, s, t\in \Z_+$ and any $y\in \Lambda_p^k$, set $|y|=k$ and denote by
$\Lambda_p^sy\Lambda_p^{t}$ the subset of $\Lambda_p^{s+k+t}$ consisting of elements of the form $wyz$
for some $w\in \Lambda_p^s$ and $z\in \Lambda_p^{t}$.

For each $n\in \Z_+$ set $B_n$ to be the subset of $\Lambda_p^{m}$ consisting of elements
$b$ for which there is no $x\in \Sigma_p$
with $x[n, n+m-1]=b$ and $x[j, j+m-1]\not \in A_j$ for all $0\le j\le n$.
Set
$C_n$ to be the subset of $\Lambda_p^{m-1}$ consisting of elements $c$ for which $\Lambda_pc$ is contained in $B_n$.
Note that $C_n$ is  exactly the set of elements $c\in \Lambda_p^{m-1}$ for which there is no $x\in \Sigma_p$
with $x[n+1, n+m-1]=c$ and $x[j, j+m-1]\not \in A_j$ for all $0\le j\le n$.
It follows that
\begin{eqnarray} \label{comb1:eq}
B_{n+1}=A_{n+1}\cup (\bigcup_{c\in C_n}c\Lambda_p).
\end{eqnarray}
Then
$$p|C_{n+1}|\le |B_{n+1}|=|A_{n+1}\cup (\bigcup_{c\in C_n}c\Lambda_p)|\le |A_{n+1}|+p|C_n|\le \ell+p|C_n|$$
for all $n\in \Z_+$. Clearly $|C_0|\le \frac{|B_0|}{p}\le \frac{\ell}{p}$.
Inductively one gets that $|C_n|\le \frac{(n+1)\ell}{p}$ for all $n\in \Z_+$.

Set $D_n$ to be the subset of $\bigcup^{m-1}_{k=0}\Lambda_p^k$ consisting of elements $y$ such that
$y\Lambda_p^{m-1-|y|}\subseteq C_n$ but $y[1, |y|-1]\Lambda_p^{m-|y|}\nsubseteq C_n$.
We put $\Lambda_p^0\subseteq D_n$ exactly when $C_n=\Lambda_p^{m-1}$.
Note that
$C_n$ is the disjoint union of $y\Lambda_p^{m-1-|y|}$ for $y\in D_n$.
For each $0\le k\le m-1$ set $D_{n, k}=\{y\in D_n: |y|=k\}$.
We claim that
$$|D_{n+1, k}|\le \ell+\sum^{m-1}_{j=k+1}|D_{n, j}|$$
for
all $n\in \Z_+$ and $0\le k\le m-1$. This is clearly true if $k=0$ or $D_{n, 0}\neq \emptyset$.
Thus assume that
$1\le k\le m-1$ and $D_{n, 0}=\emptyset$.
Let $y\in D_{n+1, k}$.
Then $y\Lambda_p^{m-1-k}$ is contained in $C_{n+1}$, and hence
$\Lambda_py\Lambda_p^{m-1-k}$
is contained in $B_{n+1}= A_{n+1}\cup (\bigcup_{z\in
D_n}z\Lambda_p^{m-|z|})$. If $\Lambda_py\Lambda_p^{m-1-k}$ has
nonempty intersection with $z_j\Lambda_p^{m-|z_j|}$, $j=1, 2,
\dots, p$, for some pairwise distinct $z_1, z_2, \dots, z_p\in
D_n$ with $\max_{1\le j\le p}|z_j|\le k$, then $\Lambda_py[1,
k-1]\Lambda_p^{m-k}$ is contained in
$\bigcup^p_{j=1}z_j\Lambda_p^{m-|z_j|}$, and hence $y[1,
k-1]\Lambda_p^{m-k}$ is contained in $C_{n+1}$, which contradicts
the assumption $y\in D_{n+1}$. Therefore
$\Lambda_py\Lambda_p^{m-1-k}$ has nonempty intersection with
$z\Lambda_p^{m-|z|}$ for at most $p-1$ elements $z\in D_n$ with
$|z|\le k$. If $\Lambda_py\Lambda_p^{m-1-k}$ does have nonempty
intersection with $z\Lambda_p^{m-|z|}$ for some $z\in D_n$ with
$|z|\le k$, then, since $|z|\ge 1$, one sees that
$\Lambda_py\Lambda_p^{m-1-k}\cap z\Lambda_p^{m-|z|}$ is equal to
$jy \Lambda_p^{m-1-k}$ for some $j\in \Lambda_p$. Therefore
$$|\Lambda_py\Lambda_p^{m-1-k}\cap (\bigcup_{z\in D_n, |z|\le k}z
\Lambda_p^{m-|z|})|\le (p-1)|y\Lambda_p^{m-1-k}|=(p-1)p^{m-1-k},$$
and hence
\begin{eqnarray} \label{comb2:eq}
|(\bigcup_{y\in D_{n+1, k}}\Lambda_py\Lambda_p^{m-1-k})\cap (\bigcup_{z\in D_n, |z|\le k}z
\Lambda_p^{m-|z|})|\le |D_{n+1, k}|(p-1)p^{m-1-k}.
\end{eqnarray}
Note that $\bigcup_{y\in D_{n+1, k}}\Lambda_py\Lambda_p^{m-1-k}\subseteq
\bigcup_{c\in C_{n+1}}\Lambda_pc\subseteq B_{n+1}$. Therefore
\begin{eqnarray} \label{comb3:eq}
\nonumber  |\bigcup_{y\in D_{n+1, k}}\Lambda_py\Lambda_p^{m-1-k}|
&=& |(\bigcup_{y\in D_{n+1, k}}\Lambda_py\Lambda_p^{m-1-k})\cap B_{n+1}|\\
\nonumber  &\overset{(\ref{comb1:eq})}=& |(\bigcup_{y\in D_{n+1, k}}
\Lambda_py\Lambda_p^{m-1-k})\cap (A_{n+1}\cup (\bigcup_{z\in D_n}z\Lambda_p^{m-|z|})|\\
\nonumber &\le & |(\bigcup_{y\in D_{n+1, k}}\Lambda_py\Lambda_p^{m-1-k})\cap (\bigcup_{z\in D_n,
|z|\le k}z\Lambda_p^{m-|z|})|\\
\nonumber & & +|(\bigcup_{y\in D_{n+1, k}}\Lambda_py\Lambda_p^{m-1-|y|})
\cap (A_{n+1}\cup (\bigcup_{z\in D_n, |z|>k}z\Lambda_p^{m-|z|}))|\\
\nonumber &\overset{(\ref{comb2:eq})}\le &|D_{n+1, k}|(p-1)p^{m-1-k}+ |A_{n+1}\cup
(\bigcup_{z\in D_n, |z|>k}z\Lambda_p^{m-|z|})|\\
\nonumber &\le &|D_{n+1, k}|(p-1)p^{m-1-k}+\ell +\sum_{z\in D_n, |z|>k}p^{m-|z|}\\
&=& |D_{n+1, k}|(p-1)p^{m-1-k}+\ell+\sum^{m-1}_{j=k+1}|D_{n, j}|p^{m-j}.
\end{eqnarray}
Since the sets $ \Lambda_py\Lambda_p^{m-1-k}$ for $y\in D_{n+1, k}$ are pairwise disjoint, we have
\begin{eqnarray} \label{comb4:eq}
|\bigcup_{y\in D_{n+1, k}}\Lambda_py\Lambda_p^{m-1-k}|=\sum_{y\in D_{n+1, k}}|
\Lambda_py\Lambda_p^{m-1-k}|=|D_{n+1, k}|p^{m-k}.
\end{eqnarray}
From (\ref{comb3:eq}) and (\ref{comb4:eq}) we get
$$ |D_{n+1, k}|p^{m-1-k}\le \ell+\sum^{m-1}_{j=k+1}|D_{n, j}|p^{m-j},$$
and hence
$$ |D_{n+1, k}|\le p^{k+1-m}\ell+\sum^{m-1}_{j=k+1}|D_{n, j}|p^{|k|+1-j}\le \ell+\sum^{m-1}_{j=k+1}|D_{n, j}|,$$
proving the claim.
It follows inductively that $|D_{n, m-k}|\le 2^{k-1}\ell$ for
all $n\in \Z_+$ and $1\le k\le m-1$.

We need to show that $B_n\neq \Lambda_p^m$ for all $n\in \Z_+$,
equivalently, $C_n\neq \Lambda_p^{m-1}$ for all $n\in \Z_+$. In
fact, we claim that for every $n\in \Z_+$, there are no $d\le
d'\le m-1$ and $y\in \Lambda_p^{m-1-d'}$ with
$y\Lambda_p^{d'}\subseteq C_n$.

We argue by contradiction. So assume that there are $n\in \Z_+$, $d\le d'\le m-1$ and
$y\in \Lambda_p^{m-1-d'}$ with $y\Lambda_p^{d'}\subseteq C_n$. Let $n_0$ be the
smallest such $n$, and let $d'$ and $y$ witness $n_0$. Replacing $y$ by $yw$ for any $w\in \Lambda_p^{d'-d}$,
we may assume that $d'=d$.

Since $\frac{(n_0+1)\ell}{p}\ge |C_{n_0}|\ge p^d>\frac{(d+p)\ell}{p}$, we have $n_0\ge d+p$. Denote $n_0-d$ by $n_1$.
For each $n_0\ge n\ge n_1$, set $E_n$ to be the subset
of $C_n$ consisting of $c$ satisfying $c[1+n_0-n, m-1-d+n_0-n]=y$. The assumption in the above paragraph says that
$E_{n_0}=y\Lambda_p^d$ and hence $|E_{n_0}|=p^d$. For each $n_0>n\ge n_1$,
we have
\begin{eqnarray*}
\bigcup_{c\in E_{n+1}}\Lambda_pc \subseteq \bigcup_{c\in C_{n+1}}\Lambda_pc\subseteq
B_{n+1}=A_{n+1}\cup (\bigcup_{c'\in C_n}c'\Lambda_p).
\end{eqnarray*}
Note that if $\Lambda_pc\cap c'\Lambda_p\neq \emptyset$ for some
$c\in E_{n+1}$ and $c'\in C_n$, then $c'$ is in $E_n$. Thus
$$\bigcup_{c\in E_{n+1}}\Lambda_pc \subseteq A_{n+1}\cup (\bigcup_{c'\in E_n}c'\Lambda_p),$$
and hence
$$p|E_{n+1}|=|\bigcup_{c\in E_{n+1}}\Lambda_pc|\le |A_{n+1}\cup (\bigcup_{c'\in E_n}c'\Lambda_p)|
\le |A_{n+1}|+p|E_n|\le \ell+p|E_n|.$$
It follows inductively that $|E_n|\ge |E_{n_0}|-\frac{(n_0-n)\ell}{p}=p^d-\frac{(n_0-n)\ell}{p}$
for all $n_0\ge n\ge n_1$.
In particular, $|E_{n_1}|\ge p^d-\frac{d\ell}{p}> \ell$.

Denote $\max(0, d+n_1-m+1)$ by $n_2$. For each $n_1\ge n\ge n_2$
denote by $k_n$ the largest number $k$ for which there exists a
subset $F$ of $C_n$ such that $|F|=k$ and $c[1+d+n_1-n, m-1]$ does
not depend on $c\in F$. Taking $F$ to be $E_{n_1}$ we see that
 $k_{n_1}\ge |E_{n_1}|> \ell$.
We claim that $pk_{n+1}\le k_n+\ell$ for all $n_1>n\ge n_2$. Take $F\subseteq C_{n+1}$  such that
$|F|=k_{n+1}$ and $c[d+n_1-n, m-1]$ does not depend on $c\in F$.  Then the set $W:=\bigcup_{c\in F}\Lambda_pc$
has $pk_{n+1}$ elements and is contained in $B_{n+1}=A_{n+1}\cup (\bigcup_{c'\in C_n}c'\Lambda_p)$.
Set $F'=\{c'\in C_n: c'\Lambda_p\cap W\neq \emptyset\}$. Then $c[1+d+n_1-n, m-1]$ does not depend
on $c\in F'$ and hence $|F'|\le k_n$. Since $d+n_1-n\le d+n_1-n_2\le m-1$, all the elements in $W$
have the same right end.
It follows that for any $c'\in F'$, one
has $|W\cap c'\Lambda_p|=1$.
Thus
\begin{eqnarray*}
pk_{n+1}&=&|W|=|W\cap (A_{n+1}\cup (\bigcup_{c'\in C_n}c'\Lambda_p))|
=|W\cap (A_{n+1}\cup (\bigcup_{c'\in F'}c'\Lambda_p))| \\
&\le &
|A_{n+1}|+\sum_{c'\in F'}|W\cap c'\Lambda_p|=|A_{n+1}|+|F'|\le \ell +k_n.
\end{eqnarray*}
This proves the claim. Inductively, we get $k_n\ge p^{n_1-n}+\ell$ for
all $n_1\ge n\ge n_2$. In particular, $|C_{n_2}|\ge k_{n_2}\ge p^{n_1-n_2}+\ell$. Since
$|C_0|\le \frac{\ell}{p}$, we have $n_2>0$.
Thus $n_2=d+n_1-m+1$, and hence $|C_{n_2}|\ge k_{n_2}\ge p^{m-d-1}+\ell$.

Since $n_2\le n_1<n_0$, according to the choice of $n_0$, $D_{n_2, k}$ is empty for all $0\le k\le m-1-d$.  Thus
\begin{eqnarray*}
|C_{n_2}|&=&|\bigcup^{m-1}_{k=0}(\bigcup_{y\in D_{n_2, k}}y\Lambda_p^{m-1-k})|\\
&=& |\bigcup^{m-1}_{k=m-d}(\bigcup_{y\in D_{n_2, k}}y\Lambda_p^{m-1-k})|\\
&=& \sum^{m-1}_{k=m-d}|D_{n_2, k}|p^{m-1-k}\\
&\le & \sum^{m-1}_{k=m-d} 2^{m-1-k}\ell \cdot p^{m-1-k}=\ell \sum^{d-1}_{j=0}(2p)^j=\frac{(2p)^d-1}{2p-1}\cdot \ell.
\end{eqnarray*}
This contradicts $|C_{n_2}|\ge p^{m-d-1}+\ell$ once we take $m$ large enough such that
$\frac{(2p)^d-1}{2p-1}\cdot \ell < p^{m-1-d}+\ell$.
A simple calculation shows that we may take $d=\ell+1$ and $m\ge 4\ell+2$.
\end{proof}

%
%
%
%
%
%


\end{document}